\documentclass[12pt]{article}
\usepackage{amsmath,amssymb,amsthm,amsfonts,amscd,euscript,verbatim, t1enc, newlfont, xypic} 
\usepackage{hyperref}
\usepackage{cancel} 
\hypersetup{breaklinks}

 \theoremstyle{definition}
\newtheorem{theo}{Theorem}[subsection]

\newtheorem{pr}[theo]{Proposition}

 \newtheorem{coro}[theo]{Corollary}
  
\theoremstyle{remark}
\newtheorem{rema}[theo]{Remark}

\theoremstyle{definition}
\newtheorem{defi}[theo]{Definition}

 \newcommand\lan{\langle}
\newcommand\ra{\rangle}
 
\newcommand\lo{\mathcal{LO}}
\newcommand\ro{\mathcal{RO}}
\newcommand\io{\mathcal{I}}
\newcommand\jo{\mathcal{J}}
\newcommand\cp{\mathcal{P}}

\newcommand\wchow{{w_{Chow}}}
\newcommand\wchowc{{w_{Chow}^c}}

\newcommand\tchow{{t_{Chow}}}

\newcommand\dmge{DM^{eff}_{gm}{}}

\newcommand\dmgm{DM_{gm}}

\newcommand\dm{DM}

\newcommand\dmc{DM^c}

\newcommand\hsing{H^{EM,\z}}
\newcommand\wsp{w^{sph}}

\newcommand\tr{Tr}

\newcommand\obj{\operatorname{Obj}}
\newcommand\mo{\operatorname{Mor}}
\newcommand\id{\operatorname{id}}
\newcommand\hu{\underline{H}}
\newcommand\cu{\underline{C}}
\newcommand\du{\underline{D}}
\newcommand\au{\underline{A}}
\newcommand\bu{\underline{B}}

\newcommand\eu{\underline{E}}

\newcommand\z{{\mathbb{Z}}}
\newcommand\q{{\mathbb{Q}}}

\newcommand\al{\alpha}

\newcommand\ns{\{0\}}

\newcommand\chow{Chow}

\newcommand\chowe{Chow^{eff}}
\newcommand\ab{Ab}
\newcommand\abfr{FAb}
\newcommand\lvect{L-\operatorname{vect}}

\newcommand\shtop{SH}

 \DeclareMathOperator\ke{\operatorname{Ker}}
 
\DeclareMathOperator\imm{\operatorname{Im}}
\DeclareMathOperator\co{\operatorname{Cone}}

\DeclareMathOperator\kar{\operatorname{Kar}}
\DeclareMathOperator\adfu{\operatorname{AddFun}}
\newcommand\hrt{{\underline{Ht}}}

\newcommand\wkar{\operatorname{Kar}^w_{\min}}

\newcommand\hw{{\underline{Hw}}}

\numberwithin{equation}{subsection}
\begin{document}

 \title{On morphisms killing weights,  weight complexes, and Eilenberg-Maclane (co)homology of spectra} 
  \author{Mikhail V. Bondarko
   \thanks{ 
Sections 1 and 2 of the paper were written under the support of the Russian Science Foundation grant no. 16-11-10200; section 3 was supported by the RFBR grant no. 15-01-03034-a and by Dmitry Zimin's Foundation ``Dynasty''.}}\maketitle

\begin{abstract}
The main goal of this paper is to study when a morphism $g$ in a triangulated category $\cu$ endowed with a weight structure  "kills certain weights" of objects (between  an integer $m$ and some $n\ge m$). 
If $g=\id_M$ (where $M\in \obj \cu$) and $\cu$ is Karoubian, then $g$ kills weights $m,\dots,n$ if and only if there exists a (weight) decomposition of $M$ that {\it avoids} these weights 
 (in the sense earlier defined by J. Wildeshaus).  

We prove the equivalence of several definitions for killing weights. In particular, we describe a family of cohomological functors that "detects" this notion. We also prove that $M$ is {\it without  weights $m,\dots, n$}  (i.e., a decomposition of $M$ avoiding these weights exists) if and only if this condition is fulfilled for its {\it weight complex} $t(M)$.

These results allow us to 
get new (stronger) results on the  conservativity  
 of the weight complex functor $t$. We study in detail the case  
$\cu=\shtop$ (endowed with the {\it spherical weight structure} whose heart consists of coproducts of sphere spectra);  the corresponding weight complex functor is just the one calculating the $H\z$-homology (whereas the terms of weight complexes are 
 free abelian groups). In this case  $g$ kills weights $m,\dots, n$ if and only if $H(g)=0$ for all $H$ represented by elements of $\shtop[m,n]$ (and 
$g$'s satisfying these conditions form an injective class in the sense essentially defined by J.D. Christensen; yet this class is not stable with respect to shifts). Moreover, for any spectrum $M$ there exists a "weakly universal decomposition" $P\to M\to I_0$ for  $I_0\in SH[m,n]$ and $P$ being without weights  $m,\dots,n$ (the latter condition has an easy description in terms of singular homology); thus we obtain a {\it torsion pair}. We also obtain a certain converse to the stable Hurewicz theorem.

\end{abstract}

\tableofcontents

 \section*{Introduction}

Recall that a weight structure $w$ (as independently defined in \cite{bws} and in \cite{konk}) is a tool for endowing objects of a triangulated category $\cu$ with certain weight filtrations; these filtrations yield  functorial ("weight") filtrations and spectral sequences for any (co)homology of these objects. 
In particular, if $w$ is (a version of) the Chow weight structure (as constructed in \cite{bws} and \cite{hebpo}; 
cf. also \S3 of \cite{bonspkar} for a certain list of weight structures of this type) for a certain motivic category then we obtain certain Deligne-type weights for \'etale and singular cohomology (see Remark 2.4.3 of \cite{bws} and Proposition 4.3.1 of \cite{bkl});  the weight structure theory also yields many more interesting filtrations and spectral sequences (including the Atiyah-Hirzebruch ones for the cohomology of spectra).

In a series of papers J. Wildeshaus has studied motives {\it without weights $m,\dots,n$} for $m\le n\in \z$. Those are motives fitting into distinguished triangles of the form $X\to M \to Y$ for $X$
  of weight at most $m-1$ and $Y$ of weight at least $n+1$ (i.e., 
$X\in \cu_{w\le m-1}$ and $Y\in \cu_{w\ge n+1}$ for $w$ being the corresponding Chow weight structure on this motivic category $\cu$). Note that $M$ determines 
its "components" $X$ and $Y$ functorially (in contrast to the case of "ordinary" weight decompositions where $n=m-1$); this yields a way of constructing  "new" (and interesting) motives out of old ones.
 If $M$ is without  weights $m,\dots,n$ (we write this as $M\in \cu_{w\notin[m,n]}$) then the corresponding factors of the weight filtration vanish for any cohomology of $M$. Somewhat amazingly, it is reasonable to expect the following interesting converse to this statement: Deligne's weights for \'etale cohomology "should detect" whether $M\in \cu_{w\notin[m,n]}$ 
 (for motives with rational coefficients; some interesting cases of this conjecture were established in \cite{wildn}, \cite{wildshim}, and  \cite{wild}). 

In the current paper we (mostly) study the condition of being without 
 weights  $m,\dots,n$ for arbitrary $(\cu,w)$. Our main tool is the study of those morphisms that {\it kill these  weights}, i.e.,  of $g\in \cu(M,N)$  "compatible with"  certain morphisms  $w_{\le n}M\to w_{\le m-1}N$ (we write this as $g\in\mo_{
\cancel{[m,n]}}\cu$).  We prove that this definition of killing weights  for $g$ is equivalent to several other ones. 
 $M$ is without weights  $m,\dots,n$ if and only if $\id_M\in\mo_{\cancel{[m,n]}}\cu$ 
 (if $\cu$ is Karoubian which is "usually" the case; we take this property  for the definition of 
 $\cu_{w\notin[m,n]}$ in general).  Certainly, if  $g$ kills weights $m,\dots,n$ then $H(g)$ kills the corresponding factors for any (co)homology $H$ on $\cu$. The converse is also true if one considers all representable $H$ here; unfortunately, \'etale cohomology is not sufficient for studying this property for motives. Yet we describe a certain class of cohomology theories (on $\cu$) such that $g\in\mo_{[\cancel{ m,n]}}\cu$ if and only if $H(g)=0$ for any theory  belonging to this collection. For $w$ being the {\it spherical} weight structure on $\shtop$ (the 
 stable homotopy category of spectra) one should take the theories represented by elements 
of $\shtop[m,n]$ (in the notation of \S3.2 of \cite{marg}, i.e., the homotopy groups of representing objects should vanish in all degrees $\notin[m,n]$; see \S\ref{sshtop} below). If $\cu=\dm$ (a "big" motivic category) then the  general theory gives representing objects belonging to $\dm[m,n]=\dm^{\tchow\le -m}\cap \dm^{\tchow\ge -n}$ (see Remark \ref{rkwmot}(3)); this is somewhat less satisfactory.

We also improve significantly our understanding of the weight complex functor $t$ (from $\cu$ into a certain "weak category of complexes" $K_w(\hw)$; weight complexes of this sort were defined in \S2.2 of \cite{bpure} following \S3 of \cite{bws}) in this paper.  Here $\hw\subset \cu$ is the additive category of objects of weight zero (i.e., $\obj \hw=\cu_{w\le 0}\cap \cu_{w\ge 0}$); so it would be reasonable to expect that $t$ (along with weight spectral sequences) "sees" the "finite weight part" of $\cu$ and ignores  "infinitely small and infinitely large weights". Our current methods yield  very precise statements of this sort; we also describe in detail the defect for $t$ to "detect" $\cu_{w\le 0}$ and $\cu_{w\ge 0}$ (see Theorem \ref{tdegen}). In particular, we prove that the singular homology of a spectrum 
 $M\in \obj \shtop$ is concentrated in non-negative degrees (in the "usual" notation; in our current notation we write $\hsing_i(M)=\ns$ for $i>0$) if and only if $M$ is an extension of a $-1$-connective spectrum by an acyclic one; this is a certain converse to the stable Hurewicz theorem (see Remark \ref{rdwss}(\ref{icsh})). 
 Moreover, $t(M)$ sees whether  $M\in \cu_{w\notin[m,n]}$. 

Now we discuss the relation of our definitions to the notion of an {\it injective  class of morphisms} (that is dual to the notion of a {\it projective class} that was central in \cite{christ}; injective classes are  "somewhat closer" to    $\mo_{\cancel{[m,n]}}\cu$ in our main examples).  By definition (see Definition \ref{dhop}(2) 
 below),  any injective class can be described as $\{f:\ H(f)=0\}$ for $H$ running through a certain family of cohomological functors $\cu\to \ab$; the same is 
  true for $\mo_{\cancel{[m,n]}}\cu$ by Theorem \ref{tkwhom}(I).  Yet in contrast to the definition of injective classes, we do not have to demand that all these $H$ are representable (one may say that the existence of weight decompositions gives a substitute for this representability condition along with the  corresponding version of the existence of enough injective objects; cf. 
	\S2.3  of  \cite{christ}). Note however that for $\cu=\dm$ or $\cu=\shtop$ (and also for $\cu=D(\au)$ for $\au$ being an abelian category with enough projectives; see Remark \ref{rtst}(5)) 
	the classes $\mo_{\cancel{[m,n]}}\cu$ are  injective (since a $t$-structure 
{\it adjacent} to $w$ exists in this case; see Proposition \ref{ptp}(2)). 
	Yet they are certainly not shift-stable in contrast to the main projective classes 
	 studied by J.D. Christensen (see the beginning of \S3 of ibid.); this distinction of our focus of study from the 
	one of ibid. is possibly even more important.
	One may say that our definition of $\mo_{\cancel{[m,n]}}\cu$   is more flexible and "takes into account filtrations". In particular, 
if 
 a $t$-structure  adjacent to $w$ exists then for any $l\ge 0$ the morphism class $\cap_{m\in \z}\mo_{\cancel{[m,m+l]}}\cu$ is shift-stable and injective; 
 yet one certainly cannot recover single 
 $\mo_{\cancel{[m,m+l]}}\cu$'s from this intersection. 
 Respectively, no reasonable analogues of 
 $\cu_{w\notin[m,n]}$ can be described using  shift-stable injective classes of $\cu$-morphisms. Still, our Theorem \ref{tprkw}(\ref{iprkwcomp}) is rather similar to Proposition 3.3 of ibid.\footnote{Note that the proof of loc. cit. actually does not require shift-stability of the corresponding ideal.} 
 Moreover, the aforementioned intersection construction yields an  interesting  (shift-stable) injective class in the case   $l=0$ 
since $\cap_{m\in \z}\mo_{\cancel{[m,m]}}\cu$ equals $\{g\in \mo\cu:\ t(g)=0\}$.

Furthermore, if a $t$-structure adjacent to $w$ exists then for any  $M\in \obj \cu$ there exists a "weakly universal decomposition" $P\to M\to I_0$ for  $I_0\in \cu[m,n]$ (see Definition \ref{dtstr}(II,III)) and $P\in \cu_{w\notin[m,n]}$. Thus $\cu_{w\notin[m,n]}$ and  $\cu[m,n]$ yield a {\it torsion pair} in the sense of \cite[Definition 1.2.2]{bpure}  (see  
see Proposition \ref{ptp}(2) below for more detail).  
This is also an important notion; in particular,  for $\cu=\shtop$ it may be interesting study the "interaction" of the ideals of morphisms characterized by the condition $\pi_i(g)=0$ (for $i$ running through $\z$; the intersection of all these classes is the class of {\it ghost morphisms}  defined in \S7 of \cite{christ}).
	
  Now we describe the contents of the paper in more detail; some more information of this sort can also be found at the beginnings of sections. 
  
  In \S\ref{sold} we recall some basics on weight structures; only a few (somewhat technical) statements are new here. The reader not much interested in compactly generated triangulated categories (such as $\shtop$), in projective and injective classes of morphisms, and in 
	 torsion  pairs may probably ignore \S\ref{svtt}. 
  
  In \S\ref{snews} we introduce our main definitions  of morphisms killing weights $m,\dots,n$ and of objects without these weights.
	For $\cu=K^b(\lvect)$ (or $=K(\lvect)$) and the stupid weight structure for this category 
	we have $g\in \mo_{\cancel{[m,n]}}\cu$ if and only if $H_i(g)=0$ 
	 for $-n\le i\le -m$ (in our numbering of homology and weights); yet the general definition is somewhat more complicated.
	We also prove 
	 several interesting properties of our notions; 
 we sometimes 
	 demand   $\cu$ to be Karoubian (i.e., if all idempotent $\cu$-endomorphisms yield  splittings of objects in it) in the formulations of this section. In particular,  our notion of being without weights $m,\dots,n$ in a non-Karoubian $\cu$  is somewhat more general than the one used by Wildeshaus.  
	We closely relate our main notions to the weight complex functor $t$; these results imply (in the Karoubian case) that $t(g)$ is an isomorphism if and only if $t(g)$ is an extension of an object of "infinitely large weight"  by a one of infinitely small weight. We (essentially) call  extensions of the latter type {\it weight-degenerate objects} and study them in detail.
	 We study in detail the case $\cu=\shtop$; for the corresponding {\it spherical} weight structure  the weight complex functor yields complexes of free abelian groups computing the $H\z$-homology of spectra, 
	cellular filtrations of spectra yield their weight Postnikov towers, and   weight spectral sequences are Atiyah-Hirzebruch ones. 
	We also relate the main subjects of this paper to  to {\it torsion pairs} and {\it injective classes} (under the assumption that there exists a $t$-structure adjacent to $w$; we recall that some statement on the existence of an adjacent $t$ were proved in \cite{bpure}).
  
	 In \S\ref{ssupl} we use the results of \cite{bonspkar} to generalizing the results of the previous section to not (necessarily) Karoubian triangulated categories.
	One of the results that we obtain is  crucial for \cite{binters}.
We also construct certain examples illustrating the distinctions between the non-Karoubian and the Karoubian case.


The author 
is deeply grateful to prof. J. Daniel Christensen and to prof. J\"org Wildeshaus for their very interesting remarks on the contents of the paper. 
He would also  like to express his gratitude to
 prof. Fernando Muro,  prof. Thomas Goodwillie, prof. Tyler Lawson, prof.   Qiaochu Yuan, and to  other users of the Mathoverflow forum for their very useful comments.

\section{Notation and a reminder on  weight structures  }\label{sold}

In this section we 
recall (a part of) the theory of  weight structures.

In \S\ref{snotata} we introduce some 
notation and 
conventions.

In \S\ref{ssws} we  recall some basics on weight structures. The only new statement of this section is Proposition \ref{pbw}(\ref{icompidemp}) (it is rather technical but quite important for this paper).

In \S\ref{sswc} we recall some 
properties of  weight complex functors. All of them except parts \ref{irwc3} and \ref{irwc6}  of Remark \ref{rwc} were  established in \cite{bws}; cf. also \S2.4 of \cite{bpure}.

In \S\ref{svtt} we discuss the relation of weight structures to cohomology; so we recall the (somewhat more complicated) notions of weight filtrations, weight range, virtual $t$-truncations, and adjacent structures. The reader not much interested in compactly generated triangulated categories (such as $\shtop$), torsion pairs, and  injective classes of morphisms may ignore this section (along with the remarks that mention it later in the text) at the first reading of the paper. 

  \subsection{Some notation and conventions}\label{snotata}
  
For categories $C,D$ we write 
$D\subset C$ if $D$ is a full 
subcategory of $C$.

Given a category $C$ and  $X,Y\in\obj C$ we  will write $C(X,Y)$ for  the set of morphisms from $X$ to $Y$ in $C$.
We will say that $X$ is  a {\it
retract} of $Y$ if $\id_X$ can be factored through $Y$. Note that if $C$ is triangulated or abelian 
then $X$ is a  retract of
$Y$ if and only if $X$ is its direct summand.

For a category $C$ the symbol $C^{op}$ will denote its opposite category.

For a subcategory $D\subset C$
we will say that $D$ is {\it Karoubi-closed} in $C$ if it contains all retracts of its objects in $C$. We will call the smallest Karoubi-closed subcategory $\kar_C(D)$ of $C$ containing $D$  the {\it Karoubi-closure} of $D$ in $C$. 

The {\it idempotent  completion} $\kar(\bu)$ (no lower index) of an additive
category $\bu$ is the category of "formal images" of idempotents in $\bu$
(so $\bu$ is embedded into an idempotent complete category). 

$\cu$ and $\cu'$ will usually denote some triangulated categories.
 We will use the
term {\it exact functor} for a functor of triangulated categories (i.e.,
for a  functor that preserves the structures of triangulated
categories).

For a distinguished triangle $A\to B\to C$  in $\cu$ we will say that $B$ is an {\it extension} of $C$ by $A$.

We will say that a class  $D\subset \obj \cu$ {\it strongly generates} a subcategory $\du\subset \cu$ and write $\du=\lan D\ra_{\cu}$ if $\du$ is the smallest  strictly full triangulated subcategory of $\cu$ such that $D\subset  \obj \du$. Certainly, here we can consider the case $\du=\cu$.


For $X,Y\in \obj \cu$ we will write $X\perp Y$ if $\cu(X,Y)=\ns$.
For $D,E\subset \obj \cu$ we  write $D\perp E$ if $X\perp Y$
 for all $X\in D,\ Y\in E$.
For $D\subset \obj \cu$ we  denote by $D^\perp$ the class
$$\{Y\in \obj \cu:\ X\perp Y\ \forall X\in D\}.$$
  Dually, ${}^\perp{}D$ is the class
$\{Y\in \obj \cu:\ Y\perp X\ \forall X\in D\}$. 

We will say that certain $C_i\in \obj \cu$ {\it Hom-generate} $\cu$ if $\{C_i[j]:\ j\in \z\}^{\perp}=\ns$.

Below $\au$ will always  denote some abelian category; $\bu$ is an additive category.

In this paper all complexes will be cohomological, i.e., the degree of
all differentials is $+1$; respectively, we will use cohomological
notation for their terms. 
We  denote by $K(\bu)$ the homotopy category
of (cohomological) complexes over $\bu$. Its full subcategory of
bounded complexes will be denoted by $K^b(\bu)$. We will write $M=(M^i)$ if $M^i$ are the terms of a complex $M$; $f^i$ will denote the $i$th component of a morphism of complexes $f$. 
If we will say that an arrow (or a sequence of arrows) in $\au$
yields an object of $K^b(\bu)$, we will mean by default  that the last object of
this sequence is in degree $0$. We will always extend a ``finite'' $\bu$-complex by $0$'s to $\pm \infty$
	(to obtain an object of $K^b(\bu)\subset K(\bu)$).

We will call a contravariant additive functor $\cu\to \au$
for an abelian $\au$ {\it cohomological} if it converts distinguished
triangles into long exact sequences. For a cohomological $F$ we will denote $F\circ [-i]$ by $F^i$.

For $I\in \obj \cu$ we will denote the cohomological functor $\cu(-,I)$ (from $\cu$ into $\ab$) by $H_I$.

On the other hand, we will call a covariant functor $F$ satisfying this condition a {\it homological} one or just homology; we will write $F_I$ for the composition $F\circ [i]$. 
So, for an $\au$-complex $(M^i,d^i:M^i\to M^{i+1})$ the object $\ke(d^{i})/\imm d^{i-1}$ is the $i$th homology $H_i(M)$. This convention is compatible with the previous papers of the author; yet it forces us to use a somewhat weird numbering for $H\z$-homology of spectra in \S\ref{sshtop}.

Let $\cu$ be a triangulated category closed with respect to  
coproducts, $B\subset \obj \cu$. 
Then an object $M$ of $\cu$ is said to be {\it compact} if the functor $\cu(M,-)$ commutes with all small coproducts.

$L$ will always be an arbitrary (fixed) field. $\lvect$ will denote the category of finite dimensional $L$-vector spaces.

\subsection{Weight structures: basics}\label{ssws}

\begin{defi}\label{dwstr}

I. A pair of subclasses $\cu_{w\le 0},\cu_{w\ge 0}\subset\obj \cu$ 
will be said to define a weight
structure $w$ for a triangulated category  $\cu$ if 
they  satisfy the following conditions.

(i) $\cu_{w\ge 0},\cu_{w\le 0}$ are 
Karoubi-closed in $\cu$
(i.e., contain all $\cu$-retracts of their objects).

(ii) {\bf Semi-invariance with respect to translations.}

$\cu_{w\le 0}\subset \cu_{w\le 0}[1]$, $\cu_{w\ge 0}[1]\subset
\cu_{w\ge 0}$.

(iii) {\bf Orthogonality.}

$\cu_{w\le 0}\perp \cu_{w\ge 0}[1]$.

(iv) {\bf Weight decompositions}.

 For any $M\in\obj \cu$ there
exists a distinguished triangle
\begin{equation}\label{wd}
X\to M\to Y
{\to} X[1]
\end{equation} 
such that $X\in \cu_{w\le 0},\  Y\in \cu_{w\ge 0}[1]$.

II. The category $\hw\subset \cu$ whose objects are
$\cu_{w=0}=\cu_{w\ge 0}\cap \cu_{w\le 0}$ and morphisms are $\hw(Z,T)=\cu(Z,T)$ for
$Z,T\in \cu_{w=0}$,
 is called the {\it heart} of 
$w$.

III. $\cu_{w\ge i}$ (resp. $\cu_{w\le i}$, resp.
$\cu_{w= i}$) will denote $\cu_{w\ge
0}[i]$ (resp. $\cu_{w\le 0}[i]$, resp. $\cu_{w= 0}[i]$).

IV.  The class $\cu_{w\ge i}\cap \cu_{w\le j}$ will be denoted by $\cu_{[i,j]}$ (so it equals $\ns$ if $i>j$).

$\cu^b\subset \cu$ will be the category whose object class is $\cup_{i,j\in \z}\cu_{[i,j]}$.

V. We will  say that $(\cu,w)$ is {\it  bounded}  if $\cu^b=\cu$ (i.e., if
$\cup_{i\in \z} \cu_{w\le i}=\obj \cu=\cup_{i\in \z} \cu_{w\ge i}$).

Respectively, we will call $\cup_{i\in \z} \cu_{w\le i}$ (resp. $\cup_{i\in \z}\cu_{w\ge
i}$) the class of $w$-{\it bounded above} (resp. $w$-{\it bounded below}) objects; we will say that $w$ is bounded above (resp. bounded below) if all the objects of $\cu$ satisfy this property.

VI. Let $\cu$ and $\cu'$ 
be triangulated categories endowed with
weight structures $w$ and
 $w'$, respectively; let $F:\cu\to \cu'$ be an exact functor.

$F$ is said to be  
{\it  weight-exact} 
(with respect to $w,w'$) if it maps
$\cu_{w\le 0}$ into $\cu'_{w'\le 0}$ and
maps $\cu_{w\ge 0}$ into $\cu'_{w'\ge 0}$.

VII. Let $\bu$ be a 
full subcategory of a triangulated category $\cu$.

We will say that $\bu$ is {\it negative} if
 $\obj \bu\perp (\cup_{i>0}\obj (\bu[i]))$.
\end{defi}

\begin{rema}\label{rstws}

1. A  simple (though quite useful for this paper) example of a weight structure comes from the stupid
filtration on 
$K(\bu)$ (or on $K^b(\bu)$) for an arbitrary additive category
 $\bu$. 
In this case
$K(\bu)_{w\le 0}$ (resp. $K(\bu)_{w\ge 0}$) will be the class of complexes that are
homotopy equivalent to complexes
 concentrated in degrees $\ge 0$ (resp. $\le 0$); see Remark 1.2.3(1) of \cite{bonspkar}.  
 
 The heart of this {\it stupid weight structure} 
is the Karoubi-closure  of $\bu$
 in 
 $K(\bu)$. 

2. A weight decomposition (of any $M\in \obj\cu$) is (almost) never canonical. 

Yet for  $m\in \z$ we will often need some choice of a weight decomposition of $M[-m]$ shifted by $[m]$. So we obtain a distinguished triangle \begin{equation}\label{ewd} w_{\le m}M\to M\to w_{\ge m+1}M \end{equation} 
with some $ w_{\ge m+1}M\in \cu_{w\ge m+1}$, $ w_{\le m}M\in \cu_{w\le m}$; we will call it an {\it $m$-weight decomposition} of $M$.  

 We will often use this notation below (though $w_{\ge m+1}M$ and $ w_{\le m}M$ are not canonically determined by $M$). Moreover, when we will write arrows of the type $w_{\le m}M\to M$ or $M\to w_{\ge m+1}M$ we will always assume that they come from some $m$-weight decomposition of $M$. 
 
3. In the current paper we use the ``homological convention'' for weight structures; 
it was previously used in \cite{wild}, \cite{hebpo}, \cite{brelmot}, \cite{wildat},  \cite{wildn}, \cite{bmm}, \cite{bpure}, 
 \cite{wildshim},  
 and \cite{bonspkar} 
 whereas in 
\cite{bws} and  \cite{bger} 
 the ``cohomological convention'' was used. In the latter convention 
the roles of $\cu_{w\le 0}$ and $\cu_{w\ge 0}$ are interchanged, i.e., one
considers   $\cu^{w\le 0}=\cu_{w\ge 0}$ and $\cu^{w\ge 0}=\cu_{w\le 0}$. So,  a
complex $X\in \obj K(\au)$ whose only non-zero term is the fifth one (i.e.,
$X^5\neq 0$) has weight $-5$ in the homological convention, and has weight $5$
in the cohomological convention. Thus the conventions differ by ``signs of
weights''; 
 $K(\au)_{[i,j]}$ is the class of retracts of complexes concentrated in degrees
 $[-j,-i]$. 
 
  4. The orthogonality axiom in Definition \ref{dwstr}(I) immediately yields that $\hw$ is negative in $\cu$.
 We will mention a certain converse to this statement  below.

\end{rema}

Let us recall some basic  properties of weight structures. 
Starting from this moment we will assume that $\cu$ is (a triangulated category) endowed with a (fixed) weight structure $w$.

\begin{pr} \label{pbw}
Let  
$m\le l\in\z$, $M,M'\in \obj \cu$, $g\in \cu(M,M')$. 

\begin{enumerate}
\item \label{idual}
The axiomatics of weight structures is self-dual, i.e., for $\du=\cu^{op}$
(so $\obj\du=\obj\cu$) there exists the (opposite)  weight
structure $w^{op}$ for which $\du_{w^{op}\le 0}=\cu_{w\ge 0}$ and
$\du_{w^{op}\ge 0}=\cu_{w\le 0}$.

 \item\label{iort}
 $\cu_{w\ge 0}=(\cu_{w\le -1})^{\perp}$ and $\cu_{w\le -1}={}^{\perp} \cu_{w\ge 1}$.

\item\label{iext} 
 $\cu_{w\le 0}$, $\cu_{w\ge 0}$, and $\cu_{w=0}$
are additive. 

\item\label{iadd} A direct sum of (a finite collection of) $m$-weight decompositions of any $M_i\in \obj \cu$ is an $m$-weight decomposition of $\bigoplus M_i$.

\item\label{icompl} 
				For any (fixed) $m$-weight decomposition of $M$ and an $l$-weight decomposition of $M$  (see Remark \ref{rstws}(2))
 $g$ can be extended 
to a 
morphism of the corresponding distinguished triangles:
 \begin{equation}\label{ecompl} \begin{CD} w_{\le m} M@>{c}>>
M@>{}>> w_{\ge m+1}M\\
@VV{h}V@VV{g}V@ VV{j}V \\
w_{\le l} M'@>{}>>
M'@>{}>> w_{\ge l+1}M' \end{CD}
\end{equation}

Moreover, if $m<l$ then this extension is unique (provided that the rows are fixed).

\item\label{icompidemp} 
Assume that we are given a diagram of the form (\ref{ecompl}) and its rows are equal (so,
$M'=M$, $m=l$,  $w_{\le m} M= w_{\le l} M'$); also suppose that $g=\id_M$ and $h$ is an idempotent endomorphism, whereas $\cu$ is Karoubian. Then 
for the decomposition $w_{\le m} M\cong M_1\bigoplus M_0$ 
corresponding to $h$ (i.e., $h$ projects $w_{\le m} M$ onto $M_1$) 
we have $M_0\in \cu_{w=m}$, whereas the upper row of (\ref{ecompl}) 
 can be presented as the direct sum of a certain $m$-weight decomposition $M_1\to M\to M_2$ and of the distinguished triangle $M_0\to 0 \to M_0[1]$.

\item\label{ifact} Assume 
 $M'\in   \cu_{w\ge m}$. Then any $g\in \cu(M,M')$ factors through $w_{\ge m}M$ (for any choice of the latter object).

 \item\label{iwd0} 
 If $M\in \cu_{w\ge m}$ 
 then $w_{\le l}M\in \cu_{[m,l]}$ (for any $l$-weight decomposition of $M$). 
Dually, if  $M\in \cu_{w\le l}$ 
 then $w_{\ge m}M\in \cu_{[m,l]}$. 

 \item\label{ifacth}
Let $M\in \cu_{w\le 0}$,  $N\in \cu_{w\ge 0}$, and fix some weight decompositions $X_1[1]{\to} M[1]\stackrel{f[1]}{\to}   Y_1[1]$ 
and $X_2\stackrel{g}{\to}   N\to Y$ of $M[1]$ and $N$, respectively. Then $Y_1,X_2\in \cu_{w=0}$ and
any  morphism from $M$ into $N$  can be presented as $g\circ h\circ f$ for some $h\in \cu(Y_1,X_2)$.
\end{enumerate}
\end{pr}
\begin{proof}

Assertions \ref{idual}, \ref{iort}, \ref{iext}, 
 \ref{icompl}, and \ref{iwd0}, 
 were proved  
in \cite{bws} (cf.  Remark 1.2.3(4) of \cite{bonspkar} and pay attention to Remark
\ref{rstws}(3) above!). 

Assertion \ref{iadd} follows from assertion \ref{iext} immediately (since  direct sums of distinguished triangles are distinguished).

To prove assertion \ref{icompidemp} we note first that $h$ does yield a certain splitting of $w_{\le m}M$ since $\cu$ is Karoubian. Next, since (\ref{ecompl}) is commutative, we obtain that $c$ factors through $M_1$. Hence the upper row of (\ref{ecompl}) can be decomposed into a direct sum of  the distinguished triangle $M_0\to 0 \to M_0[1]$ with a certain triangle $M_1\to M\to M_2$. Lastly, $M_1\in \cu_{w\le m}$ and $M_2\in \cu_{w\ge m+1}$ (since  $\cu_{w\le m}$ and $\cu_{w\ge m+1}$ are Karoubi-closed in $\cu$), whereas $M_0\in \cu_{w\le m}\cap \cu_{w\ge m+1}[-1]=\cu_{w=m}$.

Assertion \ref{ifact} follows from assertion \ref{icompl} immediately. 

Lastly, the assumptions of  assertion \ref{ifacth} imply that $Y_1,X_2\in \cu_{w=0}$ according to 
 assertion \ref{iwd0}. The rest of 
 the assertion easily  follows from assertion \ref{ifact} combined with its dual. 

\end{proof}

\begin{rema}\label{rcompl}
Diagrams of the form (\ref{ecompl}) (also in the case $l<m$) are 
 crucial for this paper.

1. An important 
type of this diagrams is the one with $g=\id_M$ (for $M'=M$; cf. part \ref{icompidemp} of the proposition). Note that for $m<l$ the corresponding connecting morphisms in  (\ref{ecompl}) are certainly unique (provided that the rows are fixed);
if $m=l$ then we obtain a certain  (non-unique) "modification" of  an $m$-weight decomposition diagram.

2. One  can "compose"  diagrams of the form (\ref{ecompl}), i.e., for any $q\in \cu(M',M'')$, $k\in \z$, and a morphism of triangles  of the form  
$$\begin{CD} w_{\le l} M'@>{}>>
M'@>{}>> w_{\ge l+1}M'\\
@VV{}V@VV{q}V@ VV{}V \\
w_{\le k} M''@>{}>>
M''@>{}>> w_{\ge k+1}M'' \end{CD}$$
one can compose its vertical arrows with the ones of (\ref{ecompl}) to obtain 
a morphism of distinguished triangles 
$$ \begin{CD} w_{\le m} M@>{c}>>
M@>{}>> w_{\ge m+1}M\\
@VV{}V@VV{q\circ g}V@ VV{}V \\
w_{\le k} M''@>{}>>
M''@>{}>> w_{\ge k+1}M'' \end{CD}$$ 
Note that one does not have to assume $k\ge l$ here (and $l\ge m$ also is not necessary provided that the existence of  (\ref{ecompl}) is known in this case).
Anyway,
 if $k>m$ then 
the composed diagram obtained this way is the only possible 
morphism of triangles compatible with $q\circ g$.

3. Note also that (\ref{ecompl}) can certainly be recovered from its left or right hand square.
\end{rema}

\subsection{On weight complexes } 
\label{sswc}

\begin{defi}\label{dwc}
For an object $M$ of $\cu$ (where $\cu$ is endowed with a weight structure $w$) choose some $w_{\le l}M$ (see Remark \ref{rstws}(2)) for all $l\in \z$.  
For all $l\in \z$ 
connect $w_{\le l-1} M$ with  $w_{\le l} M$ using Proposition \ref{pbw}(\ref{icompl}) (i.e., we consider those 
 unique connecting morphisms that are compatible with $\id_M$; see  Remark \ref{rcompl}(1)). Next,
 take the corresponding triangles 
 \begin{equation}\label{wdeck3}
  w_{\le l-1} M \to  w_{\le l} M \to M^{-l}[l]
 \end{equation}
 (so, we just introduce the notation for the corresponding cones). All of these triangles together with the corresponding morphisms   $ w_{\le l} M\to M$ are called a choice of a {\it weight Postnikov tower} for $M$, whereas the objects $M^i$ together with the morphisms connecting them (obtained by composing the morphisms $M^{-l}\to   (w_{\le l-1} M)[1-l]\to M^{-l+1}$ that come from two consecutive triangles of the type (\ref{wdeck3})) will be denoted by $t(M)$ and is said to be  a choice of a {\it weight complex} for $M$.

Respectively, for some $M,M'\in \obj\cu$, $g\in \cu(M,M')$, and some choices of their weight complexes we will say that a collection of arrows  between the terms of these 
 complexes is a choice of $t(g)$ whenever these arrows come from some morphism of the corresponding weight Postnikov towers 
 that is compatible with $g$. 

\end{defi}

Let us recall some basic properties of weight complexes.

\begin{pr}\label{pbwcomp}
Let $M,M'\in \obj \cu$, $g\in \cu(M,M')$ (where $\cu$ is endowed with a weight structure $w$).

Then the following statements are valid.

\begin{enumerate}
\item\label{iwc0} Any choice of 
$t(M)=(M^i)$ is a complex indeed (i.e., the square of the boundary is zero); all $M^i$ belong to $\cu_{w=0}$.

\item\label{iwc1} Any choice of $t(g)$ is a $C(\hw)$-morphism from the corresponding $t(M)$ to $t(M')$.

\item\label{iwc3} $M$ determines its weight complex $t(M)$ up to  homotopy equivalence.
In particular, if $M\in \cu_{w\ge 0}$, then any choice of 
$t(M)$ is $K(\hw)$-isomorphic to 
a complex with non-zero terms in non-positive degrees only; if $M\in \cu_{w\le 0}$
then $t(M)$ is isomorphic to a complex with non-zero terms in non-negative degrees only.

\item\label{iwcfu} $g$ determines its weight complex $t(g)$ up to the following {\it weak homotopy} equivalence relation: for $\hw$-complexes $A,B$ and morphisms $m_1,m_2\in C(\hw)(A,B)$  we write $m_1\backsim m_2$ if $m_1-m_2=d_Bh+jd_A$ for some collections of arrows $j^*,h^*:A^*\to B^{*-1}$.

Moreover, this equivalence relation is respected by compositions, and so considering morphisms in $K(\hw)$ modulo this relation we obtain an additive category $K_w(\hw)$.

\item\label{iwccone} 
There exist choices of  $t(M)$, $t(M')$, and (a compatible choice of) $t(g)$ such that the cone of $t(g)$ is a choice of a
weight complex of $\co (g)$.

\item\label{iwcfunct} Let $\cu'$ be a triangulated category endowed with a weight structure $w'$; let
 $F:\cu\to \cu'$ be a weight-exact functor. Then for any choice of $t(M)$ (resp. of $t(g)$) the complex $F(M^i)$ (resp. of $F_*(t(g))$) is a weight complex of $F(M)$ (resp. a choice of $t(F(g))$) with respect to $w'$.
\end{enumerate} 

\end{pr}

\begin{proof} 
 Assertions \ref{iwc0}--\ref{iwccone} follow immediately from  Theorem 3.2.2(II) and Theorem 3.3.1 of \cite{bws}. 

Assertion \ref{iwcfunct} is an immediate consequence of the definition of a weight complex (and of weight-exact functors). 
\end{proof}

\begin{rema}\label{rwc}

\begin{enumerate}
\item\label{irwc1}  The term "weight complex" originates from \cite{gs}, where a certain complex of Chow motives was constructed for a variety $X$ over a characteristic $0$ field. The weight  complex functor of Gillet and  Soul\'e can be obtained via applying the ("triangulated motivic") weight complex functor $\dmge\to K^b(\chowe)$ (or $\dmgm\to K^b(\chow)$) 
to the {\it motif with compact support} of $X$ (see Proposition 6.6.2 of \cite{mymot}). Certainly, our notion of  a weight complex functor is much more general.

\item\label{irwc2} The weak homotopy equivalence relation was introduced  in \S3.1 of \cite{bws} (independently from the earlier 
Theorem 2.1 of \cite{barrabs}). 
It has several nice properties; in particular, the identity of  a complex if weakly homotopic to $0$ if and only if this complex is contractible (see  Proposition 3.1.8(1) of \cite{bws}).

\item\label{irwc3} 
Let $\bu$ be an additive category and $k\le l\in (\{-\infty\}\cup \z\cup \{+\infty\})$; we  assume in addition that $(k,l)$ differs both from $(-\infty, -\infty)$  and from $(+\infty, +\infty)$. For $m_1,m_2: A\to B$ (for $A,B\in \obj C(\bu)$) we will write $m_1\backsim_{[k,l]}m_2$ if $m_1-m_2$ is weakly homotopic to certain $ m_0$ such that  $m_0^i=0$ for $k\le i \le l$ (and $i\in \z$). In particular, $m_1$ is weakly homotopy equivalent to $m_2$ if  $m_1\backsim_{[-\infty,+\infty]}m_2$.

Then for $m: A\to B$  we have $m\backsim_{[k,l]}0$ if there exist two sequences $h^j, g^j\in \bu(A^j,B^{j-1})$ (for $j\in \z$) such that $d_B^{j-1}\circ h^j+g^{j+1}\circ d_A^j=m^j$ for $k\le j\le l$ and $d_B^{j}\circ h^j\circ d_A^{j}=d_B^{j}\circ g^j\circ d_A^{j}$ for all $j\in \z$. Hence in the case $k=l$ we have $m \backsim_{[k,l]}0$ if and only if $m$ is homotopic to an $m_0$ such that $m_0^k=0$ (since we can take $m_0=m+d_B\circ f+f\circ d_B$, where $f^i=h^i$ for $i\le k$ and $f^i=g^i$ for $i>k$). 
Moreover  (in contrast to the homotopy equivalence relation for morphisms) the weak homotopy possesses the following property: $m\backsim_{[k,l]}0$ if and only if $m\backsim_{[i,i]}0$ for all $k\le i \le l$.
So, the weak homotopy relation has certain advantages over (the usual) homotopy one. 

Furthermore, if $\bu\subset \bu'$  and  $m_1\backsim_{[k,l]}m_2$ in $K(\bu')$ (in the notation introduced above) then $m_1\backsim_{[k,l]}m_2$ in $K(\bu)$ also. 
Lastly (according to the previous part of this remark),  if $\id_A$ is weakly homotopic to zero then $A$ is contractible. 

\item\label{irwc4} Combining these observations with Proposition 3.2.4(2) of \cite{bws} one can easily deduce the following statement. 

Adopt the notation of Proposition \ref{pbwcomp} and  fix certain choices of $t(M)$ and $t(M')$ (as well of the of the corresponding analogues of the triangles (\ref{wdeck3})).
Then the possible choices of $t(f)$ corresponding to this data form a (whole) equivalence class in $C(\hw)(t(M),t(M'))$  with respect to the weak homotopy relation.

\item\label{irwc6} 
Below we will need the following properties of $\bu$-complexes: for $A\in \obj K(\bu)$ we will write $A\backsim_w 0$ (resp. $A\backsim^w 0$) if $A$ is homotopy equivalent to a complex concentrated in non-positive  (resp. in non-negative) degrees.

Certainly, if $A\backsim_w 0$ then $\id_A \backsim_{[1,+\infty]}0_A$;  if $A\backsim^w 0$ then $\id_A \backsim_{[-\infty,-1]}0_A$.
Now we prove that the converse implications are valid also; it certainly suffices to verify the first of these statements (since the second one is its dual).

If  $\id_A \backsim_{[1,+\infty]}0_A$ then $\id_A$ is weakly equivalent to $g\in K(\bu)(A,A)$ such that $g^i=0$ for all $i>0$.
Then $g$ is an automorphism of $A$ 
 (see Proposition 3.1.8(1) of \cite{bws}) and it can be factored through the stupid truncation morphism $A \to A^{\le 0}$ (for $A^{\le 0}=\dots \to A^{-1}\to A^0\to 0\to 0\dots$). 
Hence $A$ is a $K(\bu)$-retract of $A^{\le 0}$. Lastly, 
 Theorem 3.1 of \cite{schnur} (cf. also Remark 2.1.4(2) of \cite{bsnew}) 
 yields that $A$ is homotopy equivalent to a complex concentrated in non-positive degrees indeed.

\item\label{irwc5} Our definition of weight complexes is not (quite) self-dual, since for describing the weight complex of $M\in \obj \cu$ in $\cu^{op}$ (with respect to $w^{op}$; 
see Proposition \ref{pbw}(\ref{idual})) we have to consider $w_{\ge i}M$ instead. One may say that there exist "right" and "left" weight complex functors possessing similar properties. They are actually isomorphic if $\cu$ embeds into a category that possesses a model (see Remark 1.5.9(1) of \cite{bws}); the author does not know whether this is true (for weight complexes of morphisms) in general. Yet Proposition \ref{pwckill}(\ref{iwckill3}) below 
(along with Remark \ref{rkwsd})
 demonstrate that switching to left weight complexes would not have affected 
the relation $\backsim_{[k,l]}$ for weight complexes of morphisms. 

\item\label{irwc7} 
 It appears that $t$ can "usually" be "enhanced" to an  exact functor $t_{st}:\cu\to K(\hw)$; see Corollary 3.5 of \cite{sosnwc} and \S6.3 of \cite{bws}.


\end{enumerate}

\end{rema}

\subsection{Weight filtrations, virtual $t$-truncations, and adjacent structures}\label{svtt}

Now suppose that we are given a cohomological (or just any contravariant) functor $H:\cu\to \au$, where $\au$ is an abelian category. 
We  recall that weight structures yield functorial {\it weight filtrations} for $H(-)$ (that vastly generalize the weight filtration of Deligne for  \'etale and singular homology of varieties). We also consider virtual $t$-truncations for $H$ (as defined in \S2.5 of \cite{bws} and studied in more detail in \S2 of \cite{bger}). The latter allow us to '"slice" $H$ into "more simple" pieces (of limited {\it weight range}).  
 These truncations behave as if they were given by
truncations of $H$ in some triangulated 'category of functors' $\du$
with respect to some $t$-structure (whence the name). Moreover,  this is often actually the case (in particular, in  the "motivic" and "topological" settings that will be discussed below); yet the definition does not require the existence of $\du$ (and so, does not depend on its choice). Our choice of  the numbering for them is motivated by the cohomological convention for $t$-structures (that we use in this paper following \cite{bbd}); this convention combined with the homological numbering for weight structures causes  certain (somewhat weird) $-$ signs in the definitions of this section.

\begin{defi}\label{dwfil}
Let $H:\cu\to \au$ be a 
contravariant functor, $m\le n\in \z$.

1. We define the weight filtration for $H(M)$  as $$W^m(H)(X)=\imm(H(w_{\ge m}M)\to H(M));$$ 
here we take an arbitrary choice of  $w_{\ge m}M$.

2.  We define the functor $\tau^{\ge -n }(H)$ by the correspondence $$M\mapsto\imm (H(w_{\le n+1}M)\to H(w_{\le n}M)) ;$$ here we take arbitrary choices of  
$w_{\le n}M$ and $ w_{\le n+1}M$, and connect them as in Remark \ref{rcompl}(1).

3. 
If $H$ is cohomological, we will say that it is of weight range $\ge m$ if it annihilates $\cu_{w\le m-1}$; we will say  that it is   of weight range $[m,n]$ if it also annihilates $\cu_{w\ge n+1}$.
\end{defi}

We 
recall some properties of these notions.

\begin{pr}\label{pwfil}
In the notation of the previous definition the  following statements are valid.

\begin{enumerate}

\item\label{iwfil1} $W^mH(M)$ and  $\tau^{\ge m}(H)$ are $\cu$-functorial in $M$ (for any $m$; in particular, they do not depend on the choices of the corresponding  weight decompositions of $M$).

\item\label{iwfil2} If $H$ is cohomological then  $\tau^{\ge m}(H)$ also is.

\item\label{iwfil3} The functor $\cu(-,M):\cu\to \ab$ if of weight range $\ge m$ if and only if $M\in \cu_{w\ge m}$.

\item\label{iwfil4} If $H$ is (cohomological and) of weight range $\ge m$  then  $\tau^{\ge -n}(H)$ is (also cohomological and)   of weight range $[m,n]$.

\item\label{iwfil5} If $H$ is of weight range $[m,n]$ then the morphism  $H(w_{\ge m}M)\to H(M)$ is surjective and the morphism  
$H(M)\to H(w_{\le n}M)$ is injective (here we take arbitrary choice of the corresponding weight decompositions of $M$ and apply $H$ to their connecting morphisms). 
\end{enumerate}

\end{pr}
\begin{proof}
The first part of assertion \ref{iwfil1} is given by Proposition 2.1.2(2) of \cite{bws}.

Its second part along with assertion \ref{iwfil2} is contained in 
Theorem 2.3.1 of \cite{bger}. 

The remaining assertions are contained in Proposition 2.4.4 of \cite{bpure}.

\end{proof}

Now we would like to relate virtual $t$-truncations to actual ones. We will give the definition of a $t$-structure here mostly for fixing the notation; next we define 
 adjacent weight and $t$-structures.

\begin{defi}\label{dtstr}

I. A pair of subclasses  $\cu^{t\ge 0},\cu^{t\le 0}\subset\obj \cu$
 will be said to define a
$t$-structure $t$ if 
they satisfy the
following conditions:

(i) $\cu^{t\ge 0},\cu^{t\le 0}$ are strict, i.e., contain all
objects of $\cu$ isomorphic to their elements.

(ii) $\cu^{t\ge 0}\subset \cu^{t\ge 0}[1]$, $\cu^{t\le
0}[1]\subset \cu^{t\le 0}$.

(iii)  $\cu^{t\le 0}[1]\perp
\cu^{t\ge 0}$.

(iv) For any $M\in\obj \cu$ there exists
a  {\it $t$-decomposition} distinguished triangle
\begin{equation}\label{tdec}
A\to M\to B{\to} A[1]
\end{equation} such that $A\in \cu^{t\le 0}, B\in \cu^{t\ge 0}[-1]$.

II. For $m\in n\in \z$ we will denote $\cu^{t\ge 0}[n]$ by $\cu^{t\ge -n}$; we also consider $\cu^{t\le -m}= \cu^{t\le 0}[m]$ 
and $\cu[m,n]=\cu^{t\le -m}\cap \cu^{t\ge -n}$ (cf. Definition \ref{dwstr}(IV); the $\shtop$-version of this notation can be found in  \S3.2 of \cite{marg}).

III. If $w$ is a weight structure for $\cu$ and $t$ is a $t$-structure for it we will say that $w$ is adjacent to $t$ (or that $t$ is adjacent to $w$) if $\cu^{t\le 0}=\cu_{w\ge 0}$. 

\end{defi}

We  recall  a few well-known properties of  $t$-structures and some basics on adjacent structures.

\begin{rema}\label{rtst}
Let $m\le n\in \z$; let $t$ be a $t$-structure for $\cu$.

1. Recall that the triangle (\ref{tdec}) is canonically (and functorially) determined by $M$. So for  $A'[-1-n]\to M[-1-n]\to B'[-1-n]$ being a $t$-decomposition of $M[-1-n]$  we can denote $B'$ by $t^{\ge -n}M$ (and this notation is functorial in contrast to the setting of weight structures). If $M\in \cu^{t\le -m}$ then we (certainly) have
 $t^{\ge -n}M\in \cu[m,n]$.

Recall also that $\cu^{t\ge -n}=(\cu^{t\le -1-n})^{\perp}$.

2. The heart of $t$ is defined similarly to that of $w$: this is a full subcategory $\hrt$ of $\cu$ with $\obj \hrt=\cu[0,0]$. Recall that $\hrt$ is necessarily an abelian category with short exact sequences corresponding to distinguished triangles in $\cu$.

3. The following statement is 
a particular case of Proposition 2.5.4(1) of \cite{bger} (cf. also Proposition 2.5.6(1) of ibid.):  for $H=\cu(-,M)$ we have $\tau^{\ge-n }(H)\cong \cu(-,t^{\le -n}M)$.

4. We will mostly interested in "topological" and motivic examples of weight structures. In all of these examples (see \S4.6 of \cite{bws} or \S\ref{sshtop} below for the $\shtop$-one, whereas 
a certain list of so-called Chow weight structures is described in \S3 of \cite{bonspkar}; cf. also Remark \ref{rkwmot}(3)
below) $\cu$ is closed with respect to small coproducts and there exists a small additive negative (see Definition \ref{dwstr}(VII)) subcategory $\bu\subset \cu$  such that $\bu$ Hom-generates $\cu$ and the objects of $\bu$ are compact in $\cu$. 

In this setting adjacent $w$ and $t$ for $\cu$ can be constructed as follows: 
one should take $\cu_{w\ge 0}=\cu^{t\le 0}=(\cup_{i< 0}\obj \bu[i])^{\perp}$ and recover the "remaining halves of these structures" using the corresponding orthogonality conditions  (see Theorem 4.5.2 of \cite{bws} or \S A.1 of \cite{brelmot}). Moreover, in this case $\hw$ is the idempotent completion of the category of all small coproducts of objects of $\bu$, whereas $\cu[0,0]=(\cup_{i\in \z\setminus \ns}\obj \bu[i])^{\perp}$. Furthermore, the correspondence sending $M\in \obj \cu$ into the  functor $\bu^{op}\to \ab: B\mapsto \cu(B,M)$, yields an exact equivalence of $\hrt$ with the (abelian) category of all  additive functors $\bu^{op}\to \ab$.

5. Somewhat more simple (yet certainly not "trivial") examples of adjacent structures are given by  appropriate versions of $D(\au)$ for $\au$ being an abelian category with enough projectives; then we have $\hrt\cong \au$ and $\hw\cong \operatorname{Proj} \au$ (the corresponding $w$ is obtained by considering projective hyperresolutions of $\au$-complexes).

6. A full (and functorial) description of functors of weight range $[m,m]$ is given by Proposition 2.4.4(8) of \cite{bpure}. These functors are called {\it $w$-pure} in ibid.; see Remark 2.4.5(5) of ibid. for a certain justification of this terminology.


7.  In Definition 2.1.2 of \cite{bpure} it was said that $w$ is {\bf left} adjacent to $t$ or that $t$ is {\bf right} adjacent to $w$ if $\cu_{w\ge 0}=\cu^{t\le 0}$.
Dually, $w$ was said to be right adjacent to $t$ whenever $\cu_{w\le 0}=\cu^{t\ge 0}$.


We are (currently) not interested in latter case since we do not have the corresponding $t$-structures  in the cases interesting to us. For this reason we will use the convention from Definition \ref{dtstr}(III) instead of saying that $w$ is left (or right) adjacent to  $t$.

We note however that all our definitions and results can be dualized; in particular,   weight filtrations and virtual $t$-truncations can easily be defined for homology 
(by a simple reversion of arrows; cf. Remark 2.4.4(4) of loc. cit.). 

\end{rema}

\section{On morphisms killing weights and the relation to weight complexes}\label{snews}

Recall that (a fixed triangulated category) $\cu$ is assumed to be endowed with a weight structure $w$. 

In \S\ref{sskw} we 
 define morphisms killing weights $m,\dots, n$ and  objects without these weights; we give several equivalent definitions of these notions.

In \S\ref{ssprkw} we establish several interesting properties of our notions. In particular, we prove that an object without weights  $m,\dots, n$ admits a (weight) {\it decomposition  avoiding these weights} (in the sense defined by Wildeshaus) if $\cu$ is Karoubian.
We also relate killing weights to weight filtrations for cohomology and to virtual $t$-truncations of certain representable functors; so we obtain certain
"cohomological detectors" for killing weights (and being without weights  $m,\dots, n$ for objects).

In \S\ref{skwwc} 
we relate our main notions to  the weight complex functor $t$. 
In particular, $M$ is without weights  $m,\dots, n$ if and only if  $t(M)$ 
 possesses this property.
Next we prove that $t$ is "conservative up to degenerate cones" (significantly improving the corresponding results of \S3 of \cite{bws}).

In \S\ref{sshtop} we apply our result to the study of the (topological) stable homotopy category $\shtop$ (endowed with the spherical weight structure that was defined  in \S4.6 of \cite{bws}). We fill in some gaps in the arguments of ibid. and also explain what our (new) definitions and results mean in this setting (they are closely related to the homology and cohomology given by Eilenberg-Maclane spectra). In particular, we prove the converse to  the stable Hurewicz theorem that was mentioned in the introduction.

In \S\ref{stp} we relate the main subjects of this paper to  to {\it torsion pairs} and {\it injective classes} (under the assumption that there exists a $t$-structure adjacent to $w$).

\subsection{Morphisms that kill certain weights: equivalent definitions}\label{sskw}

\begin{pr}\label{pkillw}
Let $g\in \cu(M,N)$ (for some $M,N\in \obj \cu$); $m\le n\in \z$.
Then the following conditions are equivalent.

\begin{enumerate}
\item\label{ikw1}
There exists a choice of  $w_{\le n}M$ and $w_{\ge m}N$ such that the composed morphism
$w_{\le n}M\to M\stackrel{g}{\to}N\to  w_{\ge m}N$ is zero (here the first and the third morphism in this chain come from the corresponding weight decompositions).

\item\label{ikw3}
There exists a choice of  $w_{\le n}M$ and $w_{\le m-1}N$ and of a 
morphism $h$ making the square
\begin{equation}\begin{CD} \label{ekw}
w_{\le n}M@>{}>>M
\\
@VV{
h}V@VV{g}V\\
w_{\le m-1}N@>{}>>N 
\end{CD}\end{equation}
commutative.

\item\label{ikw5}
There exists a choice of  $w_{\ge n+1}M$ and $w_{\ge m}N$ and of a 
morphism $j$ making the square
\begin{equation}\begin{CD} \label{ekw1}
M @>{}>>w_{\ge n+1}M
\\
@VV{
g}V@VV{j}V\\
N @>{}>>w_{\ge m}N 
\end{CD}\end{equation}
commutative.


\item\label{ikw7} Any choice of an $n$-weight decomposition of $M$ and an $m-1$-weight decomposition of $N$ can be completed to a morphism of distinguished triangles of the form
\begin{equation}\label{eccompl} \begin{CD} w_{\le n} M@>{}>>
M@>{}>> w_{\ge n+1}M\\
@VV{h}V@VV{g}V@ VV{j}V \\
w_{\le m-1} N@>{}>>
N@>{}>> w_{\ge m}N \end{CD}\end{equation}

\item\label{ikw8} For any choice of  $m-1$- and $n$-weight decompositions of $M$ and $N$, and for $a$ and $b$ being the corresponding (canonical) connecting morphisms 
$w_{\le m-1} M\to w_{\le n}  M$ and  $w_{\le m-1} N\to w_{\le n}  N$ respectively
(see Remark \ref{rcompl}(1)),
there exists a commutative diagram
\begin{equation}\label{edouble}
 \begin{CD} w_{\le m-1} M@>{a}>>
w_{\le n} M@>{}>> M\\
@VV{c}V@VV{d}V@ VV{g}V \\
w_{\le m-1} N@>{b}>>
w_{\le n}  N@>{}>> N \end{CD}\end{equation}
along with a morphism $h\in \cu(w_{\le n} M, w_{\le m-1} N) $ that turns the corresponding "halves" of the left hand square of (\ref{edouble}) into commutative triangles.

\item\label{ikw9} 
For any choice of the diagram (\ref{edouble}) as above its left hand commutative square can be completed to a morphism of  triangles as follows:
\begin{equation}\label{edoublecone}
 \begin{CD} w_{\le m-1} M@>{a}>>
w_{\le n} M@>{}>>\co(a)\\
@VV{c}V@VV{d}V@ VV{0}V \\
w_{\le m-1} N@>{b}>>
w_{\le n}  N@>{}>> \co(b)\end{CD}\end{equation}

\item\label{ikw9f} There exists a choice of  (\ref{edouble}) such that the corresponding diagram  (\ref{edoublecone}) is a morphism of triangles.

\end{enumerate}

\end{pr}

\begin{proof}
Conditions \ref{ikw1}, \ref{ikw3}, and \ref{ikw5} are equivalent by Proposition 1.1.9 of \cite{bbd} (that is easy; 
 in particular, the  long exact sequence $\dots\to \cu(w_{\le n}M,  w_{\le m-1}N)\to \cu(w_{\le n}M,  N)\to \cu(w_{\le n}M,  w_{\ge m}N)\dots$ yields that condition \ref{ikw1} is equivalent to \ref{ikw3}). 

Loc. cit. also implies 
 that any of these conditions  implies the existence of some diagram of the form (\ref{eccompl}) for the corresponding choices of rows.
One also obtains a diagram of this form for arbitrary choices of these weight decompositions by composing this diagram with the corresponding "change of weight decompositions" diagrams (see Remark \ref{rcompl}(1,2)); so we obtain condition \ref{ikw7}. On the other hand, the latter condition obviously implies  conditions \ref{ikw1}, \ref{ikw3}, and \ref{ikw5}.

Next, condition \ref{ikw8} certainly implies condition \ref{ikw3}. Conversely, 
obtain the commutative diagrams in condition \ref{ikw8} it suffices to take $a$ and $b$ being the canonical connecting morphisms 
$w_{\le m-1} M\to w_{\le n}  M$ and  $w_{\le m-1} N\to w_{\le n}  N$
(see Remark \ref{rcompl}(1)), $c=h\circ a$, and $d=b\circ h$. 

Next, condition  \ref{ikw9} certainly yields condition \ref{ikw9f}. 
Now, consider the long exact sequence $\dots\to \cu(w_{\le n} M,w_{\le m-1} N)\to \cu(w_{\le n} M,w_{\le n} N) \to \cu(w_{\le n} M,\co (b))\to \dots$ (for an arbitrary choice of (\ref{edouble})). If condition \ref{ikw9f} is fulfilled, the composed morphism $w_{\le n} M\stackrel{d}{\to} w_{\le n} N\to \co (b)$ is zero; hence there exists a morphism $h\in \cu(w_{\le n} M,w_{\le m-1} N)$ making the corresponding triangle (a "half" of the left hand square in (\ref{edoublecone})) commutative. Combining this with the commutativity of the right hand square in (\ref{edouble})  we obtain condition \ref{ikw3} once again.

It remains to verify that   condition \ref{ikw8} implies condition \ref{ikw9}. The aforementioned long exact sequence yields the vanishing of the corresponding composed morphism
$w_{\le n} M\to \co (b)$, whereas the long exact sequence $\dots\to \cu(w_{\le n} M,w_{\le m-1} N)\to \cu(w_{\le m-1} M,w_{\le m-1} N) \to \cu(\co(a)[-1],w_{\le m-1} N)\to \dots$ yields the vanishing of the composed morphism $\co(a)[-1]\to w_{\le m-1} N$. We obtain that (\ref{edoublecone}) is a morphism of triangles indeed.
\end{proof}


Now we give the main definitions of this paper.

\begin{defi}\label{dkw}
1. We will say that  a morphism $g$ {\it kills weights $m,\dots, n$} if it satisfies the equivalent conditions of the previous proposition (and   we will say that $f$ kills weight $m$ if $m=n$). We will denote the class of all $\cu$-morphisms killing weights  $m,\dots, n$ by $\mo_{\cancel{[m,n]}}\cu$.

2. We will say that an object $M$   is {\it without weights $m,\dots,n$} if $\id_M$ kills weights  $m,\dots, n$.
 We will denote the class of  $\cu$-objects without weights $m,\dots,n$ by  $\cu_{w\notin[m,n]}$.
\end{defi}

\begin{rema}\label{rkwsd}
1. Obviously, these definitions are self-dual in the following natural sense: 
$g\in \mo_{\cancel{[m,n]}}\cu$ (resp.  $M\in \cu_{w\notin[m,n]}$) if and only if $g$ kills $w^{op}$-weights $-n,\dots,-m$ (resp. $M$ is without $w^{op}$-weights  $-n,\dots,-m$) in $\du=\cu^{op}$ (see Proposition \ref{pbw}(\ref{idual})).

2. Now we describe a simple example that illustrates our definitions.

Let $\bu=\lvect$ (more generally, one can consider any  semi-simple  abelian category here); we endow $\cu=K(\bu)$ or $\cu=K^b(\bu)$ with the stupid weight structure  $w$ (see Remark \ref{rstws}(1)). Then  $M\in \cu_{w\le 0}$ (resp. $\in \cu_{w\ge 0}$) if and only if the homology $H_i(M)=H_0(M[i])$ (see the convention in \S\ref{snotata}) vanishes for $i<0$ (resp. for $i>0$). Hence $g\in \cu(M,N)$ kills weights $m,\dots,n$ (resp. $M$ is without weights  $m,\dots,n$) if and only if 
we have $\cu(-,K[i])(g)=0$ (resp. $\cu(M,K[i])=0$; so we put $K$ in degree $-i$)  for all $i\in \z$, $m\le i\le n$. Thus the functors $\cu(-,K[i])$ for $m\le i\le n$ yield a collection of cohomology theories that detect whether $g\in \mo_{\cancel{[m,n]}}\cu$ 
and $M\in \cu_{w\notin[m,n]}$. 
We 
 do not have so simple "detecting families" of functors in general; yet we will construct quite interesting detecting classes of cohomology below (see Theorem \ref{tkwhom} 
 for the general case and Proposition \ref{pshtopn}(\ref{it10})  for the case $\cu=\shtop$). We prefer considering cohomological detectors (in this paper) for the reasons explained in Remark \ref{rtst}(4)  (cf. also Remark \ref{rkwmot}(6) below).
\end{rema}

\subsection{Some properties of our main notions 
}\label{ssprkw}

\begin{theo}\label{tprkw}
Let  $M,N,O\in \obj \cu$, $h\in \cu(N,O)$, and assume that a morphism $g\in \cu(M,N)$
kills weights $m,\dots, n$
 for some $m\le n\in \z$.
Then
the following statements are valid.

\begin{enumerate}
\item\label{iprkw1}
Assume  $m\le m'\le n'\le n$. Then 
$g$ also kills weights $m',\dots, n'$.

\item\label{iprkwds}
$\mo_{\cancel{[m,n]}}\cu$ is closed with respect to direct sums and retracts (i.e., $\bigoplus g_i$ kills weights $m,\dots, n$ if and only if all $g_i$ do that).

\item\label{iprkwcomp} Assume that 
$h$ kills weights $m',\dots, m-1$ for some $m'<m$. Then $h\circ g$ kills weights 
$m',\dots, n$.

\item\label{iprkwfunct} Let $F:\cu\to \du$ be a weight-exact functor (with respect to a certain weight structure for $\du$) and  assume that 
 $h$ kills weights $m,\dots, n$. Then $F(h)$ kills these weights also.

\item\label{iprkwfunctemb} For $F$ and $h$ as in the previous assertion assume that $F$ is a full embedding and $F(h)\in \mo_{\cancel{[m,n]}}\du$. 
Then $h\in \mo_{\cancel{[m,n]}}\cu$. 

\item\label{iprkwcompobj} Assume that $O$ is without weights $m,\dots, n$ as well as without weights $n+1,\dots, n'$ for some $n'>n$. Then
$O\in  \cu_{w\notin[m',n]}$.

\item\label{iwildef1}
Let there exist a distinguished triangle 
\begin{equation}\label{ewild}
X\to O \to Y\to X[1]
\end{equation}
with $X\in \cu_{w\le m-1}$, $Y\in \cu_{w\ge n+1}$ (we call it a {\it  decomposition avoiding weights  $m,\dots, n$} for $M$). Then (\ref{ewild})  yields  $l$-weight decompositions of $O$ for any $l\in \z$, $m-1\le l\le n$. Moreover, $O$ is without weights $m,\dots,n$, and  
this triangle is unique up to a canonical isomorphism. 

\item\label{iwildef2} Assume that $\cu$ is Karoubian. Then the converse to the previous assertion is true also (i.e., 
any  $O$  without weights $m,\dots,n$ 
 admits a   decomposition avoiding weights  $m,\dots, n$).
\end{enumerate}
\end{theo}
\begin{proof}

\begin{enumerate}
\item Easy (if we use condition \ref{ikw3} of Proposition  \ref{pkillw}) 
 from Remark \ref{rcompl}(1,2). 

\item  Easy (if we use conditions \ref{ikw1} and \ref{ikw7} of Proposition  \ref{pkillw}); recall Proposition \ref{pbw}(\ref{iadd}). 

\item  Easy since we can compose the diagrams given by  Proposition  \ref{pkillw}(\ref{ikw3}); see Remark \ref{rcompl}(2) again.

\item 
Applying $F$ to (some choice of) the vanishing for $h$ given by condition \ref{ikw1} of Proposition  \ref{pkillw} we get this condition for $F(h)$.

\item For any choice of $w_{\le n}M$ and $w_{\ge m}N$  the composed morphism
$F(w_{\le n}M)\to F(M)\stackrel{g}{\to}N\to F( w_{\ge m}N)$ is zero (see  condition \ref{ikw3} of Proposition  \ref{pkillw}); hence  this condition is fulfilled for $h$.

\item Immediate from the previous assertion (since $\id_O\circ \id_O=\id_O$).

\item Each statement in this assertion easily follows from the previous ones.

(\ref{ewild}) yields the corresponding $l$-weight decompositions of $O$ just by definition. We obtain that $O$ is without weights $m,\dots, n$ immediately (here we can use either condition \ref{ikw3} or condition \ref{ikw5} of Proposition  \ref{pkillw}).
 This triangle (\ref{ewild})  is canonical by Proposition \ref{pbw}(\ref{icompl}) (if we take $M=M'=O$, $g=\id_O$, $m=n-1$ and $l=n$ in it). 

\item The idea is to "modify" any (fixed) $n$-decomposition  of $O$ using Proposition \ref{pbw}(\ref{icompidemp}).

We also fix an $m$-weight decomposition of $O$. By Proposition  \ref{pkillw}(\ref{ikw3}) there exists a commutative square
$$\begin{CD} 
w_{\le n}O@>{}>>O
\\
@VV{
z}V@VV{\id_O}V\\
w_{\le m-1}O@>{}>>O 
\end{CD}$$
Next, Proposition \ref{ecompl} yields the existence and uniqueness of the square
$$\begin{CD} 
w_{\le m-1}O@>{}>>O
\\
@VV{
t}V@VV{\id_O}V\\
w_{\le n}O@>{}>>O 
\end{CD}$$
Now, we can consider multiple compositions of these squares (see Remark \ref{rcompl}). Hence the aforementioned uniqueness statement yields $t=t\circ z\circ t$.
Thus the morphism $u=t\circ z$ is idempotent, and the square $$\begin{CD} 
w_{\le n}O@>{}>>O
\\
@VV{
u}V@VV{\id_O}V\\
w_{\le n}O@>{}>>O 
\end{CD}$$
is commutative. Now we apply Proposition \ref{pbw}(\ref{icompidemp});  for $X$ being the "image" of $u$ we obtain an $n$-weight decomposition 
$X\to O\to Y$. It remains to note that $X\in \cu_{w\le m-1}$ since $u$ factors through $w_{\le m-1}O$.
\end{enumerate} 
\end{proof}

\begin{rema}\label{rwild}
\begin{enumerate}
\item\label{irwild1}
The existence of  a   decomposition of $O$ avoiding weights  $m,\dots, n$ means precisely 
 that $O$  is without weights $m,\dots,n$ in the sense of Definition 1.10 of \cite{wild}. So, our definition of this notion is equivalent to the (older) definition of Wildeshaus (who introduced this term) if $\cu$ is  Karoubian. Moreover, if suffices to assume here that $\cu$ is {\it weight-Karoubian} (see Definition \ref{dpkar}(3) and Proposition \ref{pwkar} below); however,  \S\ref{ssskwild}  demonstrates that this equivalence statement does not hold unconditionally. 

Hence the uniqueness statement in Theorem \ref{tprkw}(\ref{iwildef1}) 
 coincides with Corollary 1.9 of \cite{wild}. 
Below we will also mention the functoriality of weight decompositions of this form. So we formulate the corresponding Proposition 1.7 of ibid. here:  
 if $X_i\to O_i \to Y_i$ are distinguished triangles for $i=1,2$, $X_i\in \cu_{w\le m-1}$, $Y_i\in \cu_{w\ge n+1}$ (and $n\ge m$), then any $g:O_1\to O_2$ uniquely extends to a morphism of these weight decompositions (cf. Theorem \ref{tprkw}(\ref{iwildef1}) once again; note also that the argument used in its proof extends to re-prove loc. cit. without any difficulty).  

\item\label{irwild2}
  Certainly,  $\cu_{w\notin[m,n]}$ is closed with respect to retracts and (finite) direct sums
	(use 
	part \ref{iprkwds} of the theorem; a direct proof is easy also).

\item\label{irwild4}
 Certainly, parts  \ref{iprkw1}--\ref{iprkwcomp} of  Theorem \ref{tprkw} imply that the sum of two morphisms $M\to N$ killing weights $m,\dots,n$ kills these weights also; a direct proof of this fact is  very easy as well.

\item\label{irwildeal}
Similarly to part \ref{iprkwcomp} of our theorem one can easily prove that 
$\mo_{\cancel{[m,n]}}\cu$ is a two-sided ideal of $\mo(\cu)$ (cf. \S1 of \cite{christ}), i.e., that (in addition to 
the additivity properties 
$\mo_{\cancel{[m,n]}}\cu$ that were verified above) for any morphism $j$ composable with $g$ (either from the left or from the right) the corresponding composition kills weights  $m,\dots,n$ also.  

In particular, if $M\in  \cu_{w\notin[m,n]}$ then any $\cu$-morphism from $M$ kills weights $m,\dots,n$.

Moreover, part \ref{iprkwcomp} of the theorem can certainly be re-formulated as follows: $\mo_{\cancel{[m',m-1]}}\circ \mo_{\cancel{[m,n]}}\cu\subset \mo_{\cancel{[m',n]}}\cu$ for any $m'<m$.

\item\label{irwild6} Certainly, part \ref{iprkwfunct}
of the theorem yields that weight-exact functors respect the condition of being without weights $m,\dots,n$, whereas  weight-exact full embeddings "strictly respect" this condition.
Hence weight-exact full embeddings of Karoubian categories also strictly respect the condition of an object to possess a decomposition  avoiding weights  $m,\dots, n$. 
This is also true for  weight-Karoubian categories; see 
 Proposition \ref{pwkar} below.
\end{enumerate}
\end{rema}

Now we relate 
 $\mo_{\cancel{[m,n]}}\cu$ to weights for cohomology and to virtual $t$-truncations. 
We recall that for $I\in \obj \cu$ we write $H_I$ for the cohomological functor $\cu(-,I)$ (on $\cu$).

\begin{theo}\label{tkwhom}

Let $g:M\to N$ be a $\cu$-morphism, $m\le n\in \z$.

I.  The following 
conditions are equivalent.

\begin{enumerate}
\item\label{iwf1} $g$ kills weights $m,\dots,n$.

\item\label{iwf2} $H(g)$ sends $W^{m}(H)(N)$ inside  $W^{n+1}(H)(M)$ for any contravariant functor $H:\cu\to \au$.

\item\label{iwf3} $H(g)$ sends $W^{m}(H_I)(N)$ inside  $W^{n+1}(H_I)(M)$ for 
all $I\in \cu_{w\ge m}$.

\item\label{iwf4} $H(g)=0$ for $H$ being any cohomological functor  ($\cu\to \au$) of $w$-range $[m,n]$.

\item\label{iwf5} $H(g)=0$ for $H=\tau^{\ge-n }(H_I)$ whenever $I\in \cu_{w\ge m}$.

\item\label{iwf0} $H(g)=0$, where $H=\tau^{\ge -n }(H_{I_0})$ for $I_0$ being some  fixed choice of  $w_{\ge m}N$. 

\end{enumerate}

II. The following conditions are equivalent also.
\begin{enumerate}
\item\label{iwf6} $M$ is without weights $m,\dots,n$.

\item\label{iwf7} $H(M)=0$ for $H$ being any cohomological functor ($\cu\to \au$) of $w$-range $[m,n]$.

\item\label{iwf8} $H(M)=\ns$  for $H=\tau^{\ge-n }(H_I)$ 
 whenever $I\in \cu_{w\ge m}$. 

\item\label{iwf9} $H(M)=\ns$, where $H=\tau^{\ge -n }(H_{I_0})$ for $I_0$ being some  fixed choice of  $w_{\ge m}N$. 

\end{enumerate}

\end{theo}
\begin{proof}

Certainly, condition  I.\ref{iwf2} implies condition  I.\ref{iwf3}. Next,  I.\ref{iwf4} implies condition  I.\ref{iwf5} by Proposition \ref{pwfil}(\ref{iwfil4}), and the latter condition certainly implies condition  I.\ref{iwf0} (see Proposition \ref{pwfil}(\ref{iwfil3})). 

 Now assume  that $g$ kills weights $m,\dots,n$. Then we have a commutative diagram
\begin{equation}\begin{CD} \label{ekwn}
 M @>{}>>w_{\ge n+1}M
\\
@VV{
g}V@VV{j}V\\
N @>{d}>>w_{\ge m}N 
\end{CD}\end{equation}
(it does not matter here whether we fix some choices of the rows or not). Applying $H$ to this diagram, we obtain condition  I.\ref{iwf2}.

Now fix some choice of the rows of (\ref{ekwn}) and take $I$ being (any choice of) $w_{\ge m}N$. Assume that $g$ fulfils condition  I.\ref{iwf3}; then
the morphism $d\circ g$  belongs to the image of $\cu(w_{\ge n+1}M,w_{\ge m}N)$  in $\cu(M,w_{\ge m}N)$. Thus $g$ kills weights $m,\dots,n$ (see Proposition \ref{pkillw}(\ref{ikw3})).

It remains to deduce condition  I.\ref{iwf4} from  I.\ref{iwf1}, and deduce the latter one from condition   I.\ref{iwf0}.

Assume  that $g$ kills weights $m,\dots,n$. If $H$ is a cohomological functor of $w$-range $[m,n]$ then the morphism $H(w_{\ge m}N)\to H(N)$ is surjective and the morphism  
$H(M)\to H(w_{\le n}M)$ is injective (for any choices of the corresponding weight decompositions; see Proposition \ref{pwfil}(\ref{iwfil5})).  Since the composed morphism $a:w_{\le n}M\to w_{\ge m}N$ is zero (see Proposition \ref{pkillw}(\ref{ikw7}; thus $H(a)=0$ also), we obtain  condition  I.\ref{iwf4}. 

Now assume that  condition   I.\ref{iwf0} is fulfilled. Consider the element $r$ of the group $ \tau^{\ge -n }(H_{I_0})(N)=\imm(\cu(w_{\le n+1}N,w_{\ge m}N)\to \cu(w_{\le n}N,w_{\ge m}N)) $ obtained by 
composing the corresponding connecting morphisms. Since $r$ vanishes in  $\tau^{\ge -n }(H_{I_0})(M)\subset \cu(w_{\le n}M,w_{\ge m}N)$, the composed morphism $a\in \cu(w_{\le n}M, w_{\ge m}N)$ is zero and we obtain condition  I.\ref{iwf1}.

II. Immediate from assertion I.

\end{proof}

\begin{rema}\label{rkwmot}

\begin{enumerate}
\item\label{irkwmot1}

 Certainly, if there exists a $t$-structure $t$ adjacent to $w$ (see  Definition \ref{dtstr}(II)) then the cohomology functor $H$ in parts I.\ref{iwf0} and II.\ref{iwf9} of our Theorem is isomorphic to $\cu(-,t^{\ge -n}(I_0))$ (see Remark \ref{rtst}(3)); note also that $t^{\ge -n}(I_0)\in \cu[m,n]$. We will apply this observation in \S\ref{stp} below. 

More generally, here one may consider a certain $\cu'\supset \cu$ endowed with a $t$-structure $t'$ that is {\it orthogonal} to $w$ (with respect to $\cu'(-,-)$, i.e.,
$\cu_{w\ge 0}\perp_{\cu'}\cu'^{t'\ge 1}$ and $\cu_{w\le 0}\perp_{\cu'}\cu'^{t'\le -1}$; see 
Definition 5.2.1 of \cite{bpure} or Definition 2.5.1(3) of \cite{bger});
 then the obvious analogue of this statement is valid by Proposition 2.5.4(1) of ibid.

\item\label{irkwmot2} We are (mostly) interested in "complicated" triangulated categories (yet see Remark \ref{rtst}(5)). So  for $g\in \cu(M,N)$ it can be quite difficult to check whether $g$ kills weights $m,\dots,n$ (or if 
$M\in \cu_{w\notin[m,n]}$) using (any of the versions of) the definition of these notions. Yet Theorem \ref{tkwhom} yields a way of checking that $g$  
{\bf does not} kill weights $m,\dots,n$ (resp.  $M\notin \cu_{w\notin[m,n]}$) by looking at a single  cohomology theory on $\cu$; certainly, one can choose an  "easily computable" one here.

A more complicated problem is to find a  "reasonable" detecting family of cohomology theories (see Remark \ref{rkwsd}(2)).  
Certainly, $\cu[m,n]$ yields a collection of this sort if a $t$-structure adjacent to $w$ exists (see part 1 of this remark); we will use this observation in \S\ref{sshtop} below.
Yet one can be interested in finding a smaller collection of  ("nice") detector theories that does not depend on the choice of $N$ (whereas 
 $\tau^{\ge -n }(H_{I_0})$ can be somewhat difficult to compute).

\item\label{irkwmot3} In several previous papers of the author (see also \cite{hebpo}) certain {\it Chow} weight structures 
for various motivic categories were constructed; their hearts were certain categories of Chow motives (and their coproducts). Note that in all of these examples there exists a "big" motivic category $\dm$ 
that is closed with respect to all coproducts; inside it there is its  subcategory $\dmc$ of compact objects (whose objects are usually called constructible or geometric motives). Next, inside $\dmc$ there is an (additive Karoubian) negative category $\bu$ of the corresponding Chow motives that strongly generates it. Thus we have a "geometric" Chow weight structure $\wchowc$ for $\dmc$ (whose heart is just $\bu$)  and the "big" Chow structure $\wchow$ for $\dm$ (see 
  \S2.3 of \cite{bkl} 
	or Remark \ref{rtst}(4); certainly the embedding $\dmc\to \dm$ is weight-exact).  Whereas one is usually interested in 
  $\wchowc$ only, it appears that the corresponding adjacent structure $\tchow$  (see the aforementioned remark once again) does not (usually) restrict to $\dmc$ (note also that $\hrt_{Chow}$ is isomorphic to the category of additive contravariant functors $\bu\to \ab$). Now, to detect whether a $\dmc$-morphism $g$ kills 
weights $m,\dots,n$ it suffices to compute $H_I(g)$ for $I$ running through 
$t^{\ge -n}(\dmc_{\wchow \ge m})\subset \dm[m,n]$.

\item\label{irkwmot4} Still we would certainly prefer to use some "classical" cohomology of motives instead. We try to describe the corresponding picture here.

First we recall that for motives over  a field the 
 (Chow-)weight filtrations 
for \'etale and singular cohomology (with rational coefficients) for (motives of) 
varieties differ from Deligne's ones only by a shift; see Remark 2.4.3 of \cite{bws} 
 and Proposition 4.3.1 of \cite{bkl}. Moreover, the weights of \'etale (co)homology conjecturally "detect weights" of $\q$-linear motives over an arbitrary base; see Proposition 3.3.1(4) of \cite{bmm}. 
 Yet this does not imply that \'etale (or singular) cohomology detects whether a morphism of motives kills certain weights; in particular, note that a non-zero morphism of Chow motives can certainly yield zero on cohomology. One can only hope that composing a "long" chain of motivic morphisms that kill certain 
 weights
in \'etale cohomology necessarily yields a morphism that kills weights in a certain range (in the sense of our definition). The situation is somewhat better for mixed Tate motives (see Remark \ref{rdwss}(2) below).

\item\label{irkwmot5} On the other hand, singular and \'etale cohomology of motives conjecturally detects whether a given motif is without weights $m,\dots,n$ (cf. Theorem 1.11 of \cite{wildn}); Wildeshaus has also established 
several non-conjectural cases of this statement unconditionally (see Theorem 3.4 of \cite{wildn} and Theorem 1.13(d) of \cite{wildshim}).

\item\label{irkwmot6} Applying categorical duality one certainly obtains that  
a morphism $g\in \cu(M,N)$ belongs to  $\mo_{\cancel{[m,n]}}\cu$ 
if and only if $H(g)=0$, where  $H:\cu\to \ab$ is the homological functor $O\mapsto \imm( \cu(w_{\le n}M, w_{\ge m-1}O)\to \cu(w_{\le n}M, w_{\ge m}O))$.

\item\label{irkwmot7} Certainly, part I of our theorem implies parts \ref{iprkw1} and \ref{iprkwds} of Theorem \ref{tprkw} (cf. also Remark \ref{rwild}(\ref{irwildeal})); so we obtain an alternative proof of these statements.

\end{enumerate}

\end{rema}

\subsection{Relation to the weight complex functor (and its conservativity)}\label{skwwc} 

Now we relate the properties studied in the previous subsection with the weight complex functor.

\begin{pr}\label{pwckill}
Let $g\in \cu(M,N)$ (for some $M,N\in \obj \cu$); $m\le n\in \z$.
Then the following statements are valid.

\begin{enumerate}
\item\label{iwckill1} $g$ kills weight $m$ if and only if $t(g)\backsim_{[-m,-m]} 0$ (in the notation of Remark \ref{rwc}(\ref{irwc3}); recall that this property does not depend on the choice of $t(f)$). 

\item\label{iwckill2} If 
$f_i$ for $n\ge i\ge m$ is a chain of composable $\cu$-morphism such that $t(f_i)\backsim_{[-i,-i]}0$ for all $i$, then $f_{n}\circ f_{n-1}\circ\dots \circ f_m$ kills weights  $m,\dots, n$. 

\item\label{iwckill3} $M$ is without weights $m,\dots, n$ if and only if  $t(\id_M)\backsim_{[-n,-m]}0$. 

\end{enumerate}
\end{pr}
\begin{proof}

\begin{enumerate}
\item If $g$ kills weight $m$ then we can choose $t(g)$ such that the component $t(g)^{-m}$ is zero  (see Proposition \ref{pkillw}(\ref{ikw9})). Conversely, if $t(g)\backsim_{[-m,-m]} 0$ then we can assume $t(g)^{-m}=0$ (see Remark \ref{rwc}(\ref{irwc4})); hence  $g$ kills weight $m$ (see Proposition \ref{pkillw}(\ref{ikw9f})).

\item Immediate from the previous assertion combined with Theorem \ref{tprkw}(\ref{iprkwcomp}). 

\item If $M\in \cu_{w\notin[m,n]}$ 
  then  we can choose $t(\id_M)$ so that $t(\id_M)^i= 0$ for all $i$ between $-n$ and $-m$ (this is an easy consequence of Proposition \ref{pkillw}(\ref{ikw9})); hence 
$t(\id_M)\backsim_{[-n,-m]}0$. Conversely, if  $t(\id_M)\backsim_{[-n,-m]}0$, then $t(\id_M)\backsim_{[i,i]} 0$ for all $i$ between $-m$ and $-n$; thus applying the previous assertion to 
 the composition $id_{M}^{\circ n-m+1}$ we obtain that $M$ is without weights $m,\dots, n$ . 
\end{enumerate}

\end{proof}

\begin{rema}\label{rwildwc}
1. If $M$ possesses a decomposition 
  avoiding weights  $m,\dots, n$ then Proposition \ref{pbwcomp}(\ref{iwc3})  implies that a similar decomposition exists for $t(M)$ (i.e., we can assume that the weight complexes of $X$ and $Y$ in (\ref{ewild}) are concentrated in degrees $\ge 1-m$ and $\le -1-n$, respectively). Next, if $t(M)$ possesses a decomposition 
	 satisfying this condition then  $M\in \cu_{w\notin[m,n]}$ 
	 by part \ref{iwckill3} of our proposition. 
	Hence these three conditions are equivalent if $\cu$ is Karoubian; see Theorem \ref{tprkw}(\ref{iwildef2}).

2. Certainly, part \ref{iwckill2} of our proposition is a vast generalization of (the nilpotence statement in) Theorem 3.3.1(II). 

3. 
A rich collection of triangulated categories endowed with weight structures can be described using {\it twisted complexes} (in the sense defined in \cite{bk}) over a {\it negative differential graded category} $C$ (see \S6.1--2 of \cite{bws}). We recall that for $\cu=\tr(C)$ 
a $\cu$-morphism $h$ between two twisted complexes   
$(P^i,q_{ij}),(P'^i,q'_{ij})\in \obj\cu$ is given by a certain collection of arrows $h_{ij}\in
C^{i-j}(P^i,P'^j)$ (satisfying certain "closedness" conditions and considered up to "homotopies" of a certain sort);  we have   $h_{ij}=0$ if $i>j$ since $C$ is negative. Then one can easily check that $h$ kills weights $m,\dots,n$ if and only if it is homotopic to a morphism $h'$ such that $h'_{ij}=0$ (also) if $-n\le i\le j\le m$. 

Next, the composition of morphisms in  $ \tr(C)$ is given by the "obvious" (i.e., the "matrix-like") compositions of the corresponding collections of arrows (so, it does not take the differentials in $C$ into account). Thus  Theorem \ref{tprkw}(\ref{iprkwcomp})  (cf. also Remark \ref{rwild}(\ref{irwildeal})) in this case reduces to the corresponding trivial property of (lower triangular) matrices.

Lastly, note that $t(h)$ (in this case) can be described by the collection $h_{ii}$ (for $i\in \z$); this gives an illustration for part \ref{iwckill1} of Proposition \ref{pwckill}.

4. Combining part \ref{iwckill1} of our proposition with  Theorem \ref{tkwhom}(I.\ref{iwf0}) and applying them to the stupid weight structure for $K(\bu)$ one can re-prove Theorem 2.1 of \cite{barrabs}; the corresponding detector functor coincides with the one described in loc. cit.  (cf. Remark \ref{rtst}(6)).
\end{rema}

Now we are able to improve the ("partial") conservativity property of weight complexes given by Theorem 3.3.1(V) of \cite{bws}. We need some definitions; our choice of conventions is motivated by the fact  that we consider  cohomological complexes only (in this paper).

\begin{defi}\label{dwdegen}
1. We will call the elements of $\cap_{i\in \z}\cu_{w\le i}$ (resp. of  $\cap_{i\in \z}\cu_{w\ge i}$) {\it  right degenerate} 
(resp. {\it left degenerate}).

2. We will say that $w$ is {\it non-degenerate} if $\cap_{i\in \z}\cu_{w\le i}=\cap_{i\in \z}\cu_{w\ge i}=\ns$.

3. We will say that $M\in \obj \cu$ is  {\it $w$-degenerate} (or weight-degenerate) if  $t(M)\backsim 0$.

\end{defi}

Now we relate $w$-degenerate objects to  right and left degenerate ones.

\begin{theo}\label{tdegen}
Let $g:M\to M'$ be a $\cu$-morphism. Then the following statements are valid.

I.1. $t(g)$ is an isomorphism if and only if $\co(g)$ is a $w$-degenerate object.

2. Any extension of a left   degenerate  object of $\cu$ by a  right degenerate one is $w$-degenerate.

3. If $M$ is an extension of a  left degenerate  object by an element of $\cu_{w\le 0}$ (resp. is an extension of an element of $\cu_{w\ge 0}$ by a right degenerate  object) then  $t(M)\backsim^w 0$  (resp. $t(M)\backsim_w 0$; see Remark \ref{rwc}(\ref{irwc6})). 

II. If $\cu$ is Karoubian then the statements converse to assertions I.2 and I.3 are also valid. Being more precise,     $M$ is $w$-degenerate if and only if it is an extension of a left degenerate  object by a right  degenerate one; 
 $t(M)\backsim^w 0$  (resp. $t(M)\backsim_w 0$) if and only if  $M$ is an extension of a  left degenerate   object by an element of $\cu_{w\le 0}$ (resp. is an extension of an element of $\cu_{w\ge 0}$ by a  right degenerate   object).

III. If $\cu$ contains no non-trivial left degenerate  (resp. right degenerate) objects then its class of weight-degenerate objects coincides with the one of 
 right degenerate   (resp. left degenerate) ones. In particular, if $w$ is non-degenerate then $t(M)\neq 0$ for non-zero objects.
\end{theo}
\begin{proof}
I.1. Immediate from Proposition \ref{pbwcomp}(\ref{iwc3},\ref{iwccone}).

2. If $N$ is  left (resp. right) degenerate   then we can take $w_{\le l}N=0$ (resp. $w_{\le l}N=N$) for all $l\in \z$.  
Hence  one of the choices of $t(N)$ is $0$; 
thus the same is true for extensions in question according to the previous assertion. 

3. Immediate from assertion I.1 combined with  \ref{pbwcomp}(\ref{iwc3}).

II. We investigate 
when $t(M)\backsim^w 0$. 

 For any $n>0$ we have $\id_{t(M)}\backsim_{[-n,-1]}0$ (see  Remark \ref{rwc}(\ref{irwc6})). 
 
Since $\cu$ is Karoubian, for any $n>0$ there exists a distinguished triangle $X_n\to M \to Y_n$
with $X_n\in \cu_{w\le 0}$, $Y_n\in \cu_{w\ge n+1}$ (see 
Remark \ref{rwildwc}).  All of these triangles are isomorphic to the one for $n=1$ by the uniqueness statement in Theorem \ref{tprkw}
(\ref{iwildef1}). Hence $Y_1$ is left degenerate   
and we obtain a triangle of the sort desired.

III. Certainly, the "in particular" part of the assertions follows from the (combination of the two) remaining statements.
Both of the latter follow immediately from assertion II if $\cu$ if Karoubian. 
Lastly, in the general  case if $M$ is degenerate and $\cu$ contains no non-trivial left degenerate  (resp. right degenerate) objects then $M$ is a retract of a 
 right degenerate   (resp. left degenerate) object by Corollary \ref{cwkar} bellow. 
Hence $M$ is right degenerate   (resp. left degenerate) itself.

\end{proof}

\begin{rema}\label{rcwkar} 
1. So we get a precise answer to the question when $t(g)$ is an isomorphism in the Karoubian case (in the general case one should combine part I.1 of the theorem with 
Corollary \ref{cwkar} below). In particular, the weight complex functor is conservative if $w$ is non-degenerate;
this is a significant improvement of Theorem 3.3.1(V) of \cite{bws} (that states that the restrictions of $t$ to the subcategories of left and right bounded objects of $\cu$ are conservative under this condition).

2. In Corollary \ref{cwkar}  
several equivalent conditions for $t(M)\backsim^w 0$, $t(M)\backsim_w 0$, and $t(M)=0$ (for $\cu$ being not necessarily Karoubian) will be formulated.

3. One can easily check that the motif constructed in Lemma 2.4 of \cite{ayoubconj} is $\wchow$-degenerate; here one can use the fact that $t$ commutes with countable homotopy colimits in this case 
(see Proposition 2.5.1(5) of \cite{bpure}).  On the other hand,  none of the versions of $\dm$  contains non-zero left degenerate  objects (for the  corresponding $\wchow$; see Remark \ref{rkwmot}(\ref{irkwmot3}))  since $\dm$ is Hom-generated by the corresponding Chow motives (in the sense described in \S\ref{snotata}). Thus
there exist non-zero 
 $\wchow$-right degenerate  objects in $\dm$ (at least) when the base scheme 
 is 
a big enough field and the coefficient ring is non-torsion (one can certainly generalize this to motives over other base schemes); a minor modification of the argument from  (Lemma 2.4 of) \cite{ayoubconj}  yields the corresponding example for torsion coefficients also.

\end{rema}

\subsection{The case of the spherical weight structure on the stable homotopy category} 
\label{sshtop}

Now we apply our results to the 
 stable homotopy category $\shtop$ (whose detailed description can be found in \cite{marg}; certainly, it is a Karoubian triangulated category) and the {\it spherical} weight structure for it  is defined in \S4.6 of \cite{bws}.
The sphere spectrum (in $\shtop$) will be denoted by $S^0$. We recall that the category $\bu$ of finite coproducts of 
 $S^0$ satisfies the conditions formulated in Remark \ref{rtst}(4). The corresponding $t$-structure is the Postnikov $t$-structure $t_{Post}$ whose heart is isomorphic to $\ab$; the corresponding $t_{Post}$-homology functor is given by $\pi_0=\shtop(S^0,-)$. Moreover, $\shtop^{t_{Post}\ge -n}$ (for any $n\in\z$) is the class of $n-1$-connective spectra.
 Certainly, $\shtop[0,0]$ contains exactly the Eilenberg-Maclane spectra $HG$,
where $G$ runs through all abelian groups. We will write  $H^i(M,HG)$ for $\shtop(M,HG)$. We also note that  our notation $\shtop[m,n]$ (see Definition \ref{dtstr}(II)) in this case is essentially the one introduced in  \S3.2 of \cite{marg}.

To make the notation in the current paper compatible with the one used earlier, we will use the following (somewhat weird) notation for $H\z$-homology: (for $M\in \obj \shtop$) we will denote $\shtop(S^0,M\wedge H\z[i])$ by $\hsing_i(M)$ (so, it is concentrated in negative degrees if $M$ is a connective spectrum).

Now we recall the main results of \S4.6 of \cite{bws}.

\begin{pr}\label{pshtopo}
Let $M\in \obj \shtop$, $i\in \z$. Then the following statements are valid. 

 \begin{enumerate}
\item\label{it1} $\shtop$ is endowed with a certain {\it spherical} weight structure $\wsp$ such that $\shtop_{\wsp\ge 0}=\shtop^{t_{Post}\le 0}=(\cup_{i< 0}S^0[i])^{\perp}$. 

\item\label{it2} $\shtop_{\wsp= 0}$ consists of all small coproducts of (copies of) $S^0$, whereas $\hw^{sph}$ is equivalent to $\abfr$ (the category of free abelian groups). The comparison functor is given by $\shtop(S^0,-)$; it sends $S^0$ into $\z$. 

\item\label{it3} The weight complex functor $t$  is actually  an exact functor $\shtop\to K(\abfr)\cong D(\ab)$.

\item\label{it4} $H_i(t(M))\cong \hsing_i(M)$.

\item\label{it5} $H^i(M,HG)$ is  naturally isomorphic to the $i$th homology of the complex $H(\ab(X^{-*},G))$ for  any  abelian group $G$.

\item\label{it11} $\wsp$ can be restricted to the (triangulated) subcategory $\shtop_{fin}$ of finite spectra.
\end{enumerate}
\end{pr}
\begin{proof}
These statements were proved in \S4.6 of \cite{bws} (yet see 
 the remark below).
\end{proof}

\begin{rema}\label{rmbws}
\begin{enumerate}
\item\label{imb1}
The proof of our assertion \ref{it4} in (the published version of) \cite{bws} contains a substantial gap: the spectral sequence argument used there only works if $M$ is bounded below (by Theorem 2.3.2(II(iii)) of ibid.). Note still that any $M\in \obj \shtop$ can be presented as the {\it countable homotopy colimit} of its $t_{Post}$-truncations (from above; cf.  the proof of Theorem 4.5.2(I.2) of ibid.). Since  both the left and the right hand side of the assertion yield homological functors from $SH$ into $\ab$ 
that commute with all small coproducts (the latter property is easy and well-known for $H\z$-homology, and for  $H_i(t(M))$ it follows from Proposition 2.5.1(6) of \cite{bpure}); they also respect (countable) homotopy colimits and we obtain the result.

\item\label{imb2}
Note that any object $N$ in $K_w(\abfr)\cong D(\ab)$ functorially splits as the coproduct 
of $H_i(N)[-i]$ (for $i\in \z$). Hence we can also present $N$ (in $K(\abfr)$) as 
\begin{equation}\label{epres}
\coprod_{i\in \z}(A^{i-1}\stackrel {f^i}\to B^i)[-i],\end{equation}
 where $f^i$ are certain embeddings of free abelian groups (and we put $A^j$ and $B^j$ in degree $j$ for $j\in \z$).  
Furthermore, the homology of $t(M)$  is functorially isomorphic to  $\hsing_*(M)$ (for $M\in \obj \shtop$, by Proposition \ref{pshtopo}(\ref{it4})).
Moreover, for any $g\in \shtop(M,N)$ we have  a well-defined class of $t(g)$ in 
$\ab(\hsing_i(M),\hsing_{i-1}(N))$ (for $M,N\in \obj \shtop$ and any $i\in \z$).

\item\label{imb3}
 For $X\in \obj \shtop$ and all $n\in \z$ we take $X^{(n)}$ being (any choice of) $w^{sph}_{\le n}X$. We connect $X^{(n)}\to X^{(n+1)}$ by the unique morphisms $i^n$ "compatible with $\id_X$" (see Remark \ref{rcompl}(1)); then $\co(i^n)$ is a coproduct of $S^0[n+1]$ (see Proposition \ref{pbw}(\ref{iwd0})), i.e., of  $n+1$-dimensional spheres in $\shtop$. Next, one can easily check that $X$ is the {\it minimal weak colimit} of $X^{(n)}$ (see Proposition 3.3 of \cite{marg}). Furthermore,  part  \ref{it4} of our proposition implies that the inverse limit of the $H\z$-homology of $X^{(n)}$ (with respect to $i^n$) vanishes. Hence $X^{(n)}$ give a {\it cellular tower} for $X$ (in the sense of the beginning of  \S6.3 of \cite{marg}). Conversely, if $X^{(n)}$ is (a term of) a certain cellular tower for $X$ (i.e., an {\it $n$-skeleton} of $X$ in the terms of loc. cit.) then
$X^{(n)}\in \shtop_{w^{sph}_{\le n}}$ (by Proposition \ref{pshtopn}(\ref{it7}) below) and $\co(X^{(n)}\to X)\in \shtop_{w^{sph}_{\ge n+1}}$ (since this cone is the minimal weak colimit of $\co(X^{(n)}\to X^{(l)})\in \shtop_{w^{sph}_{\ge n+1}}$ for $l\ge n$). Hence $\shtop_{\wsp\le n}$ consists exactly of {\it$n$-skeleta} (of certain spectra; cf. also Definition 6.7 of \cite{christ}) and  all possible cellular towers of $X$ come  from some choices of $w^{sph}_{\le n}X$.
 This statement was made in \S4.6 of \cite{bws}; thus we have justified it completely in the current paper. As a consequence we obtain that the  $\wsp$-{\it weight spectral sequences} (for (co)homological functors defined on $\shtop$; see \S2.3--2.4 of ibid.)  
 are actually Atiyah-Hirzebruch ones  (and we can compute them using arbitrary cellular towers; we certainly also have a similar statement for $\wsp$-filtrations).

\item\label{imb4}
Since the category $\bu$ (as well as its strong generator $S^0$)  Hom-generates $\shtop$, part \ref{it1} of the proposition yields that $\shtop$ contains no non-trivial  left degenerate  objects (i.e., that $\cap_{i\in \z}\shtop_{\wsp\ge i}=\ns$). Now we will use this observation to obtain a (less trivial) description of all degenerate objects in $\shtop$.
\end{enumerate}
\end{rema}

We use the results of the current paper to prove some new 
properties $\shtop$ (none of them were formulated in \cite{bws}; yet assertion \ref{it7} of our proposition is equivalent to Proposition 6.8 of \cite{christ} by Remark \ref{rmbws}(\ref{imb3})). 

\begin{pr}\label{pshtopn}

Let $M,N\in \obj \shtop$, $g\in \shtop(M,N)$, $m\le n\in \z$.
Then the following statements are valid.

 \begin{enumerate}
\item\label{it8} $g$ kills $w^{sph}$-weight $n$ if and only if $\hsing_{-n}(g)=0$ and 
the class of $g$ in $\ab(\hsing_{-n}(M),\hsing_{-n-1}(N))$ (see Remark \ref{rmbws}(\ref{imb2})) vanishes.

\item\label{it10} $g\in \mo_{\cancel{[m,n]}}\shtop$ 
 if and only if $H(g)=0$ for any $H$ representable by an element of $\shtop[m,n]$. Moreover, if $g$ is an $\shtop_{fin}$-morphism then it suffices to consider elements of $\shtop[m,n]$ with finitely generated $H\z$-homology here only.

\item\label{it9} $M$ is without weights $m,\dots,n$ if and only if $\hsing_i(M)=0$ for $-n\le i\le -m$ and  $\hsing_{-m+1}(M)$ is a free abelian group.

\item\label{it12} For any $M\in \obj \shtop$ there exists a distinguished triangle  $ P\stackrel{g}{\to} M\stackrel{h}{\to} I_0$ such that  $I_0\in SH[m,n]$ and $P$ is without weights  $m,\dots,n$.
Moreover,  an $\shtop$-morphism $j$  whose target is $M$ kills  weights $m,\dots, n$ if and only if it factors through $g$ (from this triangle); any morphism from $M$ into an element of  $\shtop[m,n]$ factors through (this) $h$.

\item\label{it6} The class of weight-degenerate objects of $\shtop$ is the one of  {\it acyclic spectra} (i.e., of those with vanishing $H\z$-homology); it coincides with
the class of right degenerate  spectra.

\item\label{it7} $M\in \shtop_{\wsp\le n}$ if and only if  $\hsing_i(M)=0$ for all $i\le -n$ and $\hsing_{-n+1}(M)$ is a free abelian group.

\item\label{it0}
$\hsing_i(M)=0$ for all $i> -n$ if and only if  $M$ is an extension of an object of  $ \shtop_{\wsp\ge n}=\shtop^{t_{Post}\ge -n}$ by an acyclic spectrum. Moreover,
the presentation of $M$ as an extension of this form is $\shtop$-functorial in $M$. 

 Furthermore, these two assumptions on $M$ are equivalent  to the vanishing of $H^i(M,H(\q/\z))$ for all $i<n$.


\end{enumerate}
\end{pr}
\begin{proof}
\begin{enumerate}
\item By Proposition \ref{pwckill}(\ref{iwckill1}), we should check whether $t(g)\backsim_{[-n,-n]} 0$. Thus the assertion is an easy consequence of  Remark \ref{rmbws}(\ref{imb2}).

\item The first part of the assertion is immediate from Theorem \ref{tkwhom}(I) (see also Remark \ref{rkwmot}(\ref{irkwmot1})).
We obtain the second part by noting that any finite spectrum possesses an ($m$-)$\wsp$-decomposition whose components are finite, whereas the $\hsing$-homology groups of finite spectra are finitely generated.

\item According to Proposition \ref{pwckill}(\ref{iwckill3}), we should check whether $t(\id_M)\backsim_{[-n,-m]}0$. For this purpose is suffices to apply Remark \ref{rmbws}(\ref{imb2}) again.

\item See Proposition \ref{ptp}(1) and Remark \ref{rtp}(1) below.

\item Since $\shtop$ is Karoubian and contains no non-trivial  left degenerate objects (see Remark \ref{rmbws}(\ref{imb4})), all of its weight-degenerate objects are right degenerate  (by Theorem  \ref{tdegen}(III)).  Next, since $K_w(\hw^{sph})\cong D(\ab)$, weight-degenerate spectra are exactly the acyclic ones (see  Proposition \ref{pshtopo}(\ref{it4})).

\item The proof is similar to that of assertion \ref{it9}; one should  apply  Theorem  \ref{tdegen}(II) (instead of  Proposition \ref{pwckill}(\ref{iwckill3})) and recall once again that $\shtop$ contains no non-trivial  left degenerate objects.

\item Once again, to prove the first part of the assertion we should combine Theorem  \ref{tdegen}(II) with  Remark \ref{rmbws}(\ref{imb2}). To prove the functoriality statement in question one should apply Proposition 1.7 of \cite{wild} (see Remark \ref{rwild}(\ref{irwild1}).
).

Lastly, to prove the 
"furthermore" part of the assertion it suffices to note that $\q/\z$ is an injective cogenerator of $\ab$ and apply Proposition \ref{pshtopo}(\ref{it5}).
\end{enumerate}
\end{proof}

\begin{rema}\label{rdwss}

\begin{enumerate}
\item\label{icsh} Non-zero acyclic objects do exist; see Theorem 16.17 of \cite{marg}. 

Note also that part \ref{it0} of our theorem gives a "projection" of the category of spectra satisfying the condition $\hsing_i(M)=0$ for all $i> -n$ (this is certainly the same thing as the vanishing of singular homology in degrees $<n$ if one uses the "standard" numeration convention) onto that of $n-1$-connective spectra; this projection functor does not change singular homology. Thus our result may be said to be converse to the stable Hurewicz theorem.

\item So, we obtain a "reasonable cohomological description" of 
$\mo_{\cancel{[m,n]}}\shtop$.  Moreover, if certain composable morphisms $g_n$ satisfy the (equivalent) conditions of part \ref{it8} of our proposition
for all $l\ge n\ge m$ (where $l\ge m\in \z$) 
 then  the morphism $h=g_{l}\circ g_{l-1}\circ\dots \circ g_m$ kills $\wsp$-weights  $m,\dots, l$. Next one can 
consider the corresponding version of (\ref{eccompl}) for any choices of 
cellular filtrations for the source and the target of $h$ (see
Remark \ref{rmbws}(\ref{imb3})).


\item As it often happens with (author's) results related to weight structures, we did not use the ("full") definition of $\shtop$ for proving its properties described above.
It suffices to have a category $\bu$ as in Remark \ref{rtst}(4) such the corresponding category $\hrt\cong \adfu(\bu^{op},\ab)$ is of projective dimension $1$ (in particular, this is the case if $\bu$ consists of finite coproducts of a single object $S$ and the ring $\cu(S,S)$ is  hereditary). 
Then we have $K_w(\hw)\cong D(\hrt)$ (see Remark 3.3.4 of \cite{bws}), and one can easily prove the natural analogues of all the results of this section (though  the corresponding "homology" and "cohomology" of weight complexes does not have any "topological" significance in general). 

\item In particular, inside the category $\dm$ of  motives over any perfect field 
(see \S4.2 of \cite{degmod})  one can take for $\cu$ its {\it localizing subcategory $DTM$  generated by} the Tate motives $\z(i)$ for $i\in\z$ (i.e., by its smallest triangulated subcategory containing all $\z(i)$ and closed with respect to small coproducts). Note also that this category possesses two important weight structures: the heart of the first (''Chow") one is generated by $\z(i)[2i]$ (i.e. consists of retracts of coproducts of families of $\z(i)[2i]$; actually, coproducts are not necessary here)
and this weight structure is "compatible" with Chow weight structure for the whole $\dm$, whereas  the second ("Gersten") heart is 
  generated by $\z(i)[i]$. 

One can also consider Artin-Tate motives here (cf. \cite{wildat}); instead of  motives with integral coefficients one can take $R$-linear ones  for $R$ being any Dedekind domain 
 (cf. \cite{vbook}).

\end{enumerate}
\end{rema}


\subsection{On the relation to torsion pairs and injective classes}\label{stp}

We recall some definitions from \cite{aiya}
 and \cite{christ}.

\begin{defi}\label{dhop}

1. A couple $s$ of classes $\lo,\ro\subset\obj \cu$ 
will be said to be a {\it torsion pair} (for $\cu$) if $\lo^{\perp}=\ro$,  $\lo={}^{\perp}\ro$, and 
for any $M\in\obj \cu$ there
exists a distinguished triangle
\begin{equation}\label{swd}
L_sM\stackrel{a_M}{\to} M\stackrel{n_M}{\to} R_sM
{\to} L_sM[1]\end{equation} 
such that $L_sM\in \lo $ and $ R_sM\in \ro$. We will call any triangle of this form an {\it $s$-decomposition} of $M$. 

2. We will say that a couple $(\io,\jo)$ for  $\io \subset \obj \cu$ and $\jo\subset \mo(\cu)$  is a {\it injective class} whenever the following conditions are fulfilled:

(i) For a $\cu$-morphism $g$ we have $g\in \jo$ if and only if $H_I(g)=0$ for all 
  $I\in \io$.
	
	(ii) If $M\in \obj\cu$  then the functor $H_M=\cu(-,M)$ annihilates all elements of $\jo$ if and only if $M\in \io$.

(iii) For any $M\in\obj \cu$ there
exists a distinguished triangle
\begin{equation}\label{iwd}
 JM\stackrel{j_M}{\to} M
\to IM\to 
 JM[1]\end{equation} 
such that $j_M\in \jo$ and $IM\in \io$.

\end{defi}

	\begin{rema}\label{ricl}
	1. The author has taken the term "torsion pair" from Definition 1.4 of \cite{aiya}, and used it in \cite{bpure}. Note however  in \cite[Definition 3.2]{postov} torsion pairs were called complete Hom-orthogonal pairs; the latter term appears to be somewhat less ambiguous. 
	
2. Proposition 2.6 of \cite{christ} easily implies that injective classes are categorically dual to {\it projective classes} in triangulated categories (as defined in \S2.3 of ibid.).	
	
\end{rema}

Now we relate the main notions of this paper to the ones that we have just  recalled. 

\begin{pr}\label{ptp}
Assume that $w$ is endowed with a weight structure $w$ and an adjacent $t$-structure $t$.

1. Then the couple $s=(\cu_{w\notin{[m,n]}},\cu[m,n])$ is a torsion pair.

2.  The couple $(\cu[m,n],\mo_{\cancel{[m,n]}}\cu)$ is an  injective class.


\end{pr}
\begin{proof}
1. According to Proposition 1.2.4(9) of \cite{bpure} it suffices to verify that the classes $\cu_{w\notin{[m,n]}} $ and $\cu[m,n]$ are Karoubi-closed in $\cu$, 
$\cu_{w\notin{[m,n]}} \perp \cu[m,n])$, and for any $M\in \obj \cu$ there exists an $s$-decomposition (\ref{swd}).

The class $\cu[m,n]$ is Karoubi-closed in $\cu$ since it equals the intersection of the $\cu$-Karoubi closed classes $\cu^{t\ge -n}$ and $\cu^{t\le -m}$. Moreover, $\cu_{w\notin{[m,n]}} $  is Karoubi-closed in $\cu$  according to Remark \ref{rwild}(\ref{irwild2}).

Next, $\cu_{w\notin{[m,n]}} \perp \cu[m,n]$ according to Remark \ref{rkwmot}(\ref{irkwmot1}).

It remains to verify the existence of an $s$-decomposition for an arbitrary $M\in \obj \cu$.

We fix some $w_{\ge m}M$, denote   $ t^{\ge -n}(w_{\ge m}M)$ by $RM$, and 
 complete the corresponding composed morphism $h\in \cu(M,RM)$ to a triangle $LM\to M\to RM\to LM[1]$. Then $LM$ is an extension of $t^{\le -n-1}(w_{\ge m}M)$ by $w_{\le m-1}M$ (by the octahedron axiom of triangulated categories). Since $t^{\le -n-1}(w_{\ge m}M)\in \cu_{w\ge n+1}$  (by the definition of adjacent structures), $LM$ is without weights $m,\dots, n$ (by Theorem \ref{tprkw}(\ref{iwildef1})). Lastly, since  $w_{\ge m}M\in \cu^{t\le -m}$, we have $RM\in \cu[m,n]$. 

2. Since injective classes are categorically dual to projective classes (see Remark \ref{ricl}(2)), by Proposition 1.2.4(8) of \cite{bpure} (that we apply in the dual form)  it suffices to verify condition (i) in Definition \ref{dhop}(2) for the couple $(\cu[m,n],\mo_{\cancel{[m,n]}}\cu)$. The latter follows from Theorem \ref{tkwhom}(I) 
 according to Remark \ref{rkwmot}(\ref{irkwmot1}).

\end{proof}

\begin{rema}\label{rtp}
1. Several nice properties of torsion pairs were proved in \S1.2 of ibid.; so we can apply them to our $s$.
In particular, 
Proposition 1.2.4(7) of ibid. says that for any torsion pair $(\lo,\ro)$ the class $\lo$ can be completed to a projective class (see Remark \ref{ricl}(2)), and gives a description of the corresponding morphism class. Hence  $\cu_{w\notin{[m,n]}}$ can be completed to 
 a projective class of morphisms; one may say that this class is "complementary" to  $(\cu[m,n],\mo_{\cancel{[m,n]}}\cu)$. This projective class can also be described as the {\it product} of the projective class corresponding to $\cu^{t\le -n-1}=\cu_{w\ge n+1}$ with the one corresponding to  $\cu_{w\le m-1}$ (see Proposition 3.3 of \cite{christ}).

\cite[Proposition 1.2.4(7)]{bpure} also implies that for any $s$-decomposition (\ref{swd})  any morphism into $M$ from an element of  $\cu_{w\notin{[m,n]}} $  factors through  the morphism $a_M$. Moreover, applying loc. cit. to the opposite category we obtain that any morphism from $M$ into an element of  $\cu[m,n]$ factors through $n_M$.

Furthermore, one can easily prove that a $\cu$-morphism $j$  whose target is $M$ kills  weights $m,\dots, n$ if and only if it factors through $a_M$.

 Possibly the author will study the relation  of (certain)  torsion pairs to injective classes of morphisms (including  $\mo_{\cancel{[m,n]}}\cu$) further in future. 

2. The results of ibid. yield rather vast families of examples of adjacent $w$ and $t$. One of them is provided by  Proposition 3.1.7(II) of ibid.; it implies (see Remark 3.1.8(1) of ibid.) that for any bounded weight structure $w$ on the bounded derived category $D^b(X)$ of coherent sheaves on a smooth proper variety $X$ over a field there exists an adjacent $t$. 

Now we recall some results for the case where $\cu$ is closed with respect to small coproducts. If $\cu$ also satisfies the {\it Brown representability} condition\footnote{I.e., all homological functors $\cu^{op}\to \ab$ that respect products are $\cu$-representable. Recall that this is a rather "reasonable" assumption; in particular (as proved in \cite{neebook}) it is fulfilled whenever $\cu$ is {\it perfectly generated} or $\cu^{op}$ is compactly generated. } then to any {\it smashing} $w$ on $\cu$ (i.e., we assume that $\cu_{w\ge 0}$ is closed with respect to $\cu$-coproducts) there exists an adjacent $t$ according to Theorem 3.1.2 of ibid.
Recall also that any {\it perfect} set $\cp$ of objects of $\cu$ {\it generates} a smashing weight structure; see Theorem 4.3.1, Definition 3.3.1(2), and Remark 2.1.5(1) of ibid. In particular, one can take $\cp$ being any set of compact objects. 


3. Another related notion is the natural $t$-structure analogue of killing weights; so, for $t$ being a $t$-structure for $\cu$ and $g\in\cu(M,N)$ one may ask whether the corresponding composed morphism $t^{\le n}M\to N$ factors through $t^{\le m-1}M$ (for $m\le n\in \z$). 
Note that this setting  is closely related to ghost morphisms as studied in \S7 of \cite{christ}.


4. Assume that $\cu$ is closed with respect to small products.\footnote{Actually, it appears that it suffices to assume the existence of countable products only; yet this distinction does not seem to be important since any $\cu$ satisfying the Brown representability condition is also closed with respect to products according to Proposition 8.4.6  of \cite{neebook}.} Combining Proposition 3.1 of \cite{christ} with part 2 of our proposition one obtains that for any $n\ge 0$ the couple $(\io_n, \cap_{m\in \z} \mo_{\cancel{[m,m+n]}}\cu)$ is an injective class, where $\io_n$ is the set of retracts of all products of elements of the class $\cup_{m\in \z}\cu[m,m+n]$.

\end{rema}

\section{On  generalizations to non-Karoubian categories and applications} 
\label{ssupl}

This section is mainly dedicated to the extension of our main results to non-Karoubian triangulated categories and to their applications.

In \S\ref{spkar} we discuss certain extensions of the  results of 
 \S\ref{snews} to the case where $\cu$ is not (necessarily) Karoubian (they are mostly "generalizations up to retracts"). 

In \S\ref{snkarex} we construct certain counterexamples to demonstrate that the modifications made in \S\ref{spkar} to "adjust" the results of \S\ref{snews}
to the non-Karoubian case 
cannot be avoided. 

\subsection{On weight-Karoubian extensions and generalizations of our results to non-Karoubian categories}\label{spkar}

We recall the central definitions of \cite{bonspkar}.

\begin{defi}\label{dpkar}
1. We will call a triangulated category $\cu'$ an {\it  idempotent extension} of $\cu$ if  it contains $\cu$ and there exists a fully faithful exact functor  $\cu'\to \kar(\cu)$\footnote{Recall that (according to Theorem 1.5 of \cite{bashli}) the category  $\kar(\cu)$ can be naturally endowed with the structure of a triangulated category so that the  natural embedding functor $\cu\to \kar(\cu)$ is exact. Hence
$\cu'$ is an idempotent extension of $\cu$ if  and only if any object of $\cu'$ is a retract of some object of $\cu$ and $\cu$ is dense (see \S\ref{snotata})  in $\cu'$.}.

2. We will say that a weight structure $w$  {\it extends} onto an  idempotent extension $\cu'$ of $\cu$ 
 whenever there  exists a weight structure 
 $w'$ for $\cu'$ such that the embedding $\cu\to \cu'$ is weight-exact.  In this case we will call $w'$ an {\it extension} of $w$. 

3. We will say that a triangulated category $\cu'$ endowed with a weight structure $w'$ is {\it weight-Karoubian} if $\hw'$ is Karoubian. 

4. We will call  a weight-Karoubian 
category $(\cu',w')$ a   {\it weight-Karoubian extension} of $(\cu,w)$ if $\cu'$ is an  idempotent extension of $\cu$ and $w'$ is the extension of $w$ onto it.
\end{defi}

Now we recall those results of ibid. that are relevant for the current paper.

\begin{pr}\label{ppkar}
1. Let $\cu'$ be an  idempotent extension of $\cu$ such that  $w$ for extends to a weight structure $w'$ on it. Then  $\cu'_{w\ge 0}$ (resp.  $\cu'_{w'\le 0}$, resp.  $\cu'_{w'= 0}$)  is the Karoubi-closure of  $\cu_{w\ge 0}$ (resp.  $\cu_{w\le 0}$, resp.  $\cu_{w= 0}$) in $\cu'$.

2. Any $(\cu,w)$ possesses a weight-Karoubian extension.

\end{pr}
\begin{proof}
1. This is Theorem 2.2.2(I.1) of ibid.

2. The statement is given part III.1 of loc. cit.
\end{proof}

The following observations is crucial for this section.

\begin{pr}\label{pwkar}
 The 
conclusions of Proposition \ref{pbw}(\ref{icompidemp}), Theorem \ref{tprkw}(\ref{iwildef2}), and Theorem  \ref{tdegen}(II)
remain valid if we assume $\cu$ to be weight-Karoubian (only).
\end{pr}
\begin{proof}

It suffices to verify that the first of the  statements mentioned can be generalized this way, since then the proofs of the other two facts (given above) would extend to the weight-Karoubian case automatically.

The idea is to construct the retracts mentioned in  Proposition \ref{pbw}(\ref{icompidemp}) inside $\kar(\cu)$, and  prove then that they are actually isomorphic to objects of $\cu$.
So, we consider the 
$\kar(\cu)$-decomposition  $w_{\le m} M\cong M_1\bigoplus M_0$ corresponding to $h$. Since $\hw$ is Karoubian, it suffices to verify that $M_0$ is a retract of $M_m=w_{\ge m}(w_{\le m} M)\in \cu_{w=m}$ (see part \ref{ifact} of the proposition).  Since $M_0$ is a retract of $M_{w\ge m+1}[-1]\in \cu_{w\ge m}$, we have $\cu_{w\le m-1}\perp M_0$.
Thus 
if we apply the functor 
$\kar(\cu)(-,M_0)$ to the distinguished triangle 
$ w_{\le m-1}M (=w_{\le m-1}(w_{\le m}M))\to w_{\le m}M\to M_0$  then the resulting long exact sequence  yields that the projection $w_{\le m} M\to M_0$ factors through $M_m$. Certainly, $\id_{M_0}$ possesses this property also.
\end{proof}

\begin{rema}\label{rwkar}

 In  \cite{bonspkar} much more information on  idempotent extensions of $\cu$ such that $w$ extends to them is contained. In particular, the  (essentially) minimal weight-Karoubian extension of $\cu$ was described as follows: $\wkar(\cu)=\lan \obj \kar(\cu^-)\cup \obj \kar(\cu^+) \ra_{\kar(\cu)}$. Since it is minimal,
applying our proposition to it gives
the maximal possible amount of information on the corresponding $\cu$.
\end{rema}

Now we use Proposition \ref{pwkar} for deducing a certain version of Theorem  \ref{tdegen}(II) that would be valid for a not (necessarily) Karoubian $\cu$.

\begin{coro}\label{cwkar}
Let $M\in \obj \cu$. 

I. 
The following conditions are equivalent.
\begin{enumerate}
\item\label{icwk1} $M$ is weight-degenerate (resp. $t(M)\backsim^w 0$). 

\item\label{icwk2}  $M$ can be presented as an extension of a left $w$-degenerate  object of $\cu$ by a  right degenerate   one
(resp. by an element of $\cu'_{w'\le 0}$) in some weight-Karoubian extension $\cu'$ of $\cu$.

\item\label{icwk3} Such a decomposition of $M$ exists in any weight-Karoubian extension of $\cu$.

\item\label{icwk4} $M$ is a $\cu$-retract of an extension of a left degenerate   object of $\cu$ by a  right degenerate   one
(resp. by  an element of $\cu_{w\le 0}$).

\item\label{icwk5} The object $M\bigoplus M[-1]$ is an extension of this sort. 

\end{enumerate}

II. The following conditions are equivalent also.

\begin{enumerate}
\item\label{icwk6} $t(M)\backsim_w 0$. 

\item\label{icwk6r} $t(M)$ is a retract of a complex  concentrated in non-positive  degrees (in $K(\hw)$). 

\item\label{icwk7}  $M$ can be presented as an extension of an element of $\cu'_{w'\ge 0}$ by a  right degenerate  object 
  in some weight-Karoubian extension $\cu'$ of $\cu$.

\item\label{icwk8} Such a decomposition of $M$ exists in any weight-Karoubian extension of $\cu$.

\item\label{icwk9} $M$ is a $\cu$-retract of an extension 
of an element of $\cu_{w\ge 0}$ by a  right degenerate  object.

\item\label{icwk0} The object $M\bigoplus M[1]$ is an extension 
of 
this sort.
\end{enumerate}

\end{coro}
\begin{proof} We will only prove assertion II; the proof of assertion I is similar.

Certainly, condition \ref{icwk0} of the assertion implies condition \ref{icwk9}. 
\ref{icwk7} follows from \ref{icwk8} since a weight-Karoubian extension $(\cu',w')$ of $\cu$ exists (see Proposition \ref{ppkar}(2)). 

Next we note that (for any  weight-Karoubian extension $\cu'$ of $\cu$ and a fixed $M$)
$t(M)\backsim_w 0$ in $K(\hw)$ if and only if this is true in $K(\hw')$ (see Proposition \ref{pbwcomp}(\ref{iwcfunct}) and 
Remark \ref{rwc}(\ref{irwc3},\ref{irwc6})). Hence  condition \ref{icwk9} implies condition \ref{icwk6}.
Moreover, \ref{icwk6}is equivalent to \ref{icwk6r} by Remark \ref{rwc}(\ref{irwc6}).

Next we fix some $(\cu',w')$  and recall that (the conclusion of)
Theorem  \ref{tdegen}(II) 
can be applied to $\cu'$ according to 
 Proposition \ref{pwkar}(1).
 Hence  condition \ref{icwk6}  implies condition \ref{icwk8}. 

It remains to deduce condition \ref{icwk0} from condition \ref{icwk7}. 
For any $N'\in \obj \cu'$ being the "formal image" of an idempotent $p\in \cu(N,N)$ (for some $N\in \obj \cu$) we have $\co(p)\cong N'\bigoplus N'[1]\in \obj \cu$ (cf. Lemma 2.2 of \cite{thom}). Hence the direct sum of the "decomposition" of $N$ given by condition \ref{icwk7} with its shift by $[1]$ yields condition  \ref{icwk0}.
\end{proof}

\begin{rema}\label{rwkarloc}
Now we describe some consequences of our results that are used in \cite{binters}.

Firstly, if $\cu$ 
is left non-degenerate then part I of our corollary certainly implies that any its $w$-degenerate object is right degenerate. 

Now assume in addition that there is a weight-exact functor $F:\du\to \cu$, where $\du$ is a triangulated category endowed with a weight structure $v$. 
Consider two choices $(M^i_1)$  and  $(M^i_2)$ of $v$-weight complexes of an object $M$ of $\du$. If $F$ kills all $\du$-morphisms from $M^i_1$  and  $M^i_2$ for all $i\in \z$ then we certainly have 
$t_{w}(F(\id_M))=0$. Hence $F(M)$ is a right degenerate object of $\cu$.

We will apply this statement for "computing intersections" of triangulated $\du_1,\du_2\subset \du$. We assume that $v$ restricts to 
  $\du_1$ and $\du_2$ (so, this yields the corresponding choices of $t_{v}(M)$ for $M\in \obj \du_1\cap \obj \du_2$). Thus if we assume in addition that a weight-exact $\du\to\cu$ annihilates all $\du$-morphisms from  $\hu_1$ into $\hu_2$ (where $\hu_i$ are the hearts of the weight structures for $\du_i$) then we will obtain that $F(M)$ is right degenerate in $\cu$ for any $M\in \obj \du_1\cap \obj \du_2$. In particular, if $M$ is also $v$-bounded below then $F(M)=0$. We will use this statement for $F$ being the Verdier localization functor of $\du$ by its Karoubi-closed subcategory $\du_3$; so we obtain that $M$ essentially (i.e., up to an isomorphism) belongs to $ \obj \du_3$ (if $M$ is  $v$-bounded below). 

\end{rema}

\subsection{Some counterexamples in the non-Karoubian case}\label{snkarex}

Our examples will be rather simple; their main "ingredient" is $K(\lvect)$ (the homotopy category of 
 complexes of finite dimensional $L$-vector spaces; here $L$ is an arbitrary fixed field).

\subsubsection{An "indecomposable" weight-degenerate object }\label{sindwd}
 
Now we demonstrate that Theorem  \ref{tdegen}(II) does not extend to arbitrary (i.e., to not necessarily weight-Karoubian) triangulated categories.

Our example will be the full subcategory $\cu$ of $(K^b(\lvect))^3$ consisting of objects whose "total Euler characteristic"  is even (i.e., the sum of dimensions of all cohomology of all the three components of $M=(M_1,M_2,M_3)$ should be even). We define $w$ for $\cu$ as follows: $\cu_{w\le 0}$ consists of those $(M_1,M_2,M_3)$ such that $M_1\cong 0$ and $M_2$ is acyclic in negative degrees; 
$(M_1,M_2,M_3)\in \cu_{w\ge 0}$ if $M_1\cong 0$ and  $M_2$ is acyclic in positive degrees. This is easily seen to be a weight structure (in particular, a weight decomposition of $(M_1,M_2,M_3)$ is $(0,M',M_3)\to  (M_1,M_2,M_3)\to (M_1,M'',0)$, where $M'\to M_2\to M''$ is a stupid weight decomposition of $M_2$ in 
$K^b(\lvect)$ with the corresponding parity of the Euler characteristics). Next, one can easily see that the object $M=(L,0,L)$ (here we put the $L$'s in degree $0$ though the degrees make no difference) is weight-degenerate (since it is weight-degenerate in the obvious extension of $w$ onto its weight-Karoubian extension $\cu'=(K^b(\lvect))^3$; see Proposition \ref{pbwcomp}(\ref{iwcfunct})). Yet $M$ certainly cannot be presented as an extension of a left degenerate  object (i.e., of an object whose last two components are zero) by a 
an element of $\cu_{w\le 0}$ (since the corresponding "decomposition" in $\cu'$ is unique and its "components" have odd "total Euler characteristics"). So, we obtain that first two statements in Theorem  \ref{tdegen}(II) do not extend to $\cu$; for the same reasons, the third statement in loc. cit. does not extend to $\cu$ either (and the same $M$ does not possess the corresponding "decomposition"). 

Looking at the proof Theorem  \ref{tdegen}(II) one immediately obtains the existence of $n>0$ such that 
there does not exist a triangle $X_n\to M \to Y_n$
with $X_n\in \cu_{w\le -n}$, $Y_n\in \cu_{w\ge n}$. Moreover, one can easily check directly that a triangle of this sort does not exist for $n=1$ already.

\subsubsection{A bounded object that  is  without weight $0$ but does not possess a decomposition avoiding this weight}\label{ssskwild}


So, the example above yields that Theorem \ref{tprkw}(\ref{iwildef2}) does not extend to arbitrary $(\cu,w)$ (i.e., that our definition of objects without weights $m,\dots,n$ is not equivalent to Definition 1.10 of \cite{wild} in general). Yet the weight structure is degenerate in this example. Now we give a bounded 
example of the non-equivalence of definitions.
Denote by $\bu$  the category of even-dimensional vector spaces over 
$L$; take $\cu=K^b(\bu)$, $M=L^2\stackrel{\begin{pmatrix}
 1 & 0  \\
 0 & 0 
\end{pmatrix}}{\to} L^2\stackrel{\begin{pmatrix}
 0 & 0  \\
 0 & 1 
\end{pmatrix}}{\to} L^2$; we put these vector spaces in degrees $-1,0$, and $1$, respectively. 
Certainly, the composition $(L^2\to L^2\to 0)\to M\to (0\to L^2\to L^2)$ is zero; so, $M$ is without weight $0$ (see Proposition \ref{pkillw}(\ref{ikw1})).
 Yet $M$ does not possess a  decomposition 
avoiding weight $0$ since the $L$-Euler characteristics of the corresponding
$X$ and $Y$ cannot be odd.

Certainly, this example also yields that 
 decompositions avoiding weights  $m,\dots, n$ do not "lift" from a (weight-)Karoubian $\cu'$ (in our case $\cu'=K^b(\lvect)$; the corresponding weight structure is the stupid one) to $\cu$ (cf. Remark \ref{rwild}(\ref{irwild6})). 

\end{document}